\documentclass[12pt,a4paper,fleqn]{article}
\input{packages}
\usepackage{lmodern} 
\usepackage[right = 2cm, top=2.5cm, bottom=2.5cm,left=2cm]{geometry}
\definecolor{dgreen}{rgb}{0,0.5,0}
\definecolor{dblue}{rgb}{0,0,0.5}
\definecolor{dred}{rgb}{0.6,0.0,0.1}
\definecolor{dgold}{rgb}{0.5,0.3,0.0}
\definecolor{dvio}{rgb}{0.6,0.3,0.5}
\definecolor{gray}{rgb}{0.5,0.5,0.5}
\definecolor{dbraun}{rgb}{.5,0.2,0}

\newcommand{\colre}{dred}

\newcommand{\colrem}{dgold}
\newcommand{\colil}{dgreen}
\newcommand{\Lp}[1][]{\sL^{#1}}
\newcommand{\Vnorm}[2][]{\lVert#2\rVert_{#1}}   
   
\newcommand{\VnormLp}[2][{\Lp}^2]{\Vnorm[{#1}]{#2}}   
\newcommand{\VnormInf}[2][\infty]{\Vnorm[{#1}]{#2}} 
\def\var{\mathop{\mathrm{var}}\nolimits}%
\newcommand{\Nsuite}[2][j]{(#2)_{#1 \in\Nz}}

\newcommand{\set}[1]{{\left\lbrace #1\right\rbrace }}

\newcommand{\floor}[1]{\lfloor #1\rfloor}

\newcommand{\lra}{\longrightarrow} 

\renewcommand{\subset}{\subseteq}
\newcommand{\IN}{\mathbb{N}}
\newcommand{\IQ}{\mathbb{Q}}
\newcommand{\IZ}{\mathbb{Z}}
\newcommand{\IR}{\mathbb{R}}
\newcommand{\IC}{\mathbb{C}}

\newcommand{\pRz}[1][\Rz]{#1_{+}} 
\newcommand{\IP}{\mathbb{P}}
\newcommand{\IE}{\mathbb{E}}
\newcommand{\iid}{\overset{\text{iid}}{\sim}}

\newcommand{\lcb}{\left\lbrace} 
\newcommand{\rcb}{\right\rbrace} 
\newcommand{\lv}{\left\vert} 
\newcommand{\rv}{\right\vert} 
\newcommand{\lV}{\left\Vert} 
\newcommand{\rV}{\right\Vert} 
\newcommand{\lb}{\left(} 
\newcommand{\rb}{\right)} 
\newcommand*{\mc}[1]{\mathcal{#1}}

\newcommand{\dif}{\text{d}}

\newcommand{\Ind}[1]{\mathds{1}\mbox{\scriptsize$#1$}}
\declaretheoremstyle[
    spaceabove=10pt, 
    spacebelow=6pt, 
    headfont=\color{\colre}\normalfont\bfseries,
    notefont=\mdseries\bfseries, 
    notebraces={(}{)}, 
    bodyfont=\normalfont\itshape,
    postheadspace=.3em,
    headpunct={.}]{restyle}

\declaretheoremstyle[
    spaceabove=8pt, 
    spacebelow=8pt, 
    headfont=\color{\colrem}\normalfont\bfseries,
    notefont=\mdseries\bfseries, 
    notebraces={(}{)}, 
    bodyfont=\normalfont\itshape,
    postheadspace=.3em,
    qed=\smaller$\color{\colrem}\square$, 
    headpunct={.}]{remstyle}

\declaretheoremstyle[
    spaceabove=8pt, 
    spacebelow=8pt, 
    headfont=\color{\colil}\normalfont\bfseries,
    notefont=\mdseries\bfseries, 
    notebraces={(}{)}, 
    bodyfont=\normalfont\itshape,
    postheadspace=.3em,
    qed=\smaller$\color{\colil}\square$, 
    headpunct={.}]{ilstyle}

\declaretheorem[name=Theorem, style=restyle, numberwithin=section]{theorem}

\declaretheorem[name=Lemma, style=restyle, numberlike=theorem]{lemma}
\declaretheorem[name=Proposition, style=restyle, numberlike=theorem]{proposition}
\declaretheorem[name=Corollary, style=restyle, numberlike=theorem]{corollary}
\declaretheorem[name=Remark, style=remstyle, numberlike=theorem]{remark}

\declaretheorem[name=Illustration, style=ilstyle, numberlike=theorem]{illustration}

\newcommand{\Nz}{{\mathbb N}}

\newcommand{\sL}{\mathscr{L}}

\newcommand{\oA}{\overline{A}}

\newcommand{\uA}{\underline{A}}

\newcommand{\tA}{\widetilde{A}}

\newcommand{\nb}{k}

\newcommand{\xOb}[1][\nb]{X_{#1}}
\newcommand{\yOb}[1][\nb]{Y_{#1}}
\newcommand{\epsOb}[1][\nb]{\varepsilon_{#1}}

\newcommand{\xden}[1][]{f_{#1}}
\newcommand{\xoden}[1][]{f_{\circ {#1}}}
\newcommand{\yden}[1][]{g_{#1}}
\newcommand{\epsden}[1][]{\varphi_{#1}}

\newcommand{\hyden}[1][]{\widehat{g}_{#1}}

\newcommand{\ccon}{\text{\textcircled{$\star$}}}

\newcommand{\say}{n}

\newcommand{\qF}{\mathrm{q}^2} 
\newcommand{\hqF}{\hat{\mathrm{q}}^2} 
\newcommand{\qFr}{{\mathrm{q}}} 
\newcommand{\tqF}{\widetilde{\mathrm{q}}^2} 
\newcommand{\test}{\Delta}
\newcommand{\sera}{\rho}
\newcommand{\msera}[1][]{\sera_\star}
\newcommand{\optdim}{\kappa_\star}
\newcommand{\optdimest}[1][]{\kappa_{{#1}\star}}
\newcommand{\mrate}[1][]{r_{{#1}\star}}
\newcommand{\baseterm}[1][]{r^4_{{#1}\circ}}

\newcommand{\densities}{\mc D}

\newcommand{\rwdclass}{\mc E_{{\wdclass[]}}^{\rdclass}}
\newcommand{\expan}{j}
\newcommand{\wdclass}[1][\expan]{a_{#1}}
\newcommand{\rdclass}{R}

\newcommand{\x}{\xi}
\newcommand{\y}{a_{m}}
\newcommand{\expb}[1][\expan]{e_{#1}}

\newcommand{\Pf}[1][f]{\IP_{#1}} 
\newcommand{\Ef}[1][f]{\IE_{#1}} 
\newcommand{\Vf}[1][f]{\var_{#1}}

\newcommand{\erisk}[1][]{r^2(#1)}
\newcommand{\trisk}[1][]{\mc R(#1) }

\newcommand{\TV}{\mathrm{TV}}
\newcommand{\chisq}{\chi^2}
\newcommand{\signv}[1][]{\tau_{#1}}
\newcommand{\ssconst}{\mathfrak a} 

\newcommand{\sPara}{s}
\newcommand{\pPara}{p}

\newcommand{\hela}{\mathrm{h}}
\newcommand{\hell}{\mathrm{H}}
\newcommand{\qfun}{\mathbbm{q}^2}
\newcommand{\pfun}{\mathbbm{p}^2}

\author{{\sc Sandra Schluttenhofer}\;\thanks{Institut f\"ur Angewandte
    Mathematik, M$\Lambda$THEM$\Lambda$TIKON, Im Neuenheimer Feld 205,
  D-69120 Heidelberg, Germany, e-mail:
  \url{{schluttenhofer|johannes}@math.uni-heidelberg.de}} \and {\sc Jan Johannes}$\;^*$}
\date{Ruprecht-Karls-Universität Heidelberg} 
\title{Minimax testing and quadratic functional estimation for circular convolution} 
\begin{document}
\maketitle
\begin{abstract}
In a circular convolution model, we aim to infer on the density of a circular random variable using observations contaminated by an additive measurement error. We highlight the interplay of the two problems: optimal testing and quadratic functional estimation. Under general regularity assumptions, we determine an upper bound for the minimax risk of estimation for the quadratic functional. The upper bound consists of two terms, one that mimics a classical bias-variance trade-off and a second that causes the typical elbow effect in quadratic functional estimation. Using a minimax optimal estimator of the quadratic functional as a test statistic, we derive an upper bound for the nonasymptotic minimax radius of testing for nonparametric alternatives.  Interestingly, the term causing the elbow effect in the estimation case vanishes in the radius of testing. We provide a matching lower bound for the testing problem. By showing that any lower bound for the testing problem also yields a lower bound for the quadratic functional estimation problem, we obtain a lower bound for the risk of estimation. Lastly, we prove a matching lower bound for the term causing the elbow effect in the estimation problem.
The results are illustrated considering Sobolev spaces and ordinary or super smooth error densities. 
\end{abstract}
{\footnotesize
	\begin{tabbing} 
		\> circular data, deconvolution, quadratic functionals,  goodness-of-fit
	\\[.2ex] 
\end{tabbing}}%
\section{Introduction}
\paragraph*{The statistical model.}
We consider a circular convolution model, where a random variable that takes values on the circle is observed contaminated by an additive error. The aim of this paper is to highlight the interplay of the two problems: optimal testing and quadratic functional estimation for its density. Identifying the circle with the unit interval $[0,1)$, the observable random variable is
\begin{align}
	\label{model}
	\yOb[] = \xOb[] + \epsOb[] - \floor{\xOb[] + \epsOb[]},
\end{align}
where $\xOb[]$ and $\epsOb[]$ are independent random variables supported on the interval $[0,1)$ and $\floor{\cdot}$ denotes the floor-function. Let $\xOb[]$ be distributed with the unknown density of interest $\xden$ and the error $\epsOb[]$ with the known density $\epsden$. The density $\yden$ of the observable random variable $\yOb[]$ satisfies $\yden = \xden \ccon \epsden$, where $\ccon$ denotes circular convolution defined by
\begin{align*}
	\xden \ccon \epsden (y) = \int_{0}^1 \xden\lb y - s - \floor{y-s} \rb \epsden(s) \dif s, \qquad y \in [0,1).
\end{align*}
Hence, making inference on $\xden$ based on observations from $\yden$ is a deconvolution problem.

\begin{figure}
	\centering
	\includegraphics[width=0.4\textwidth]{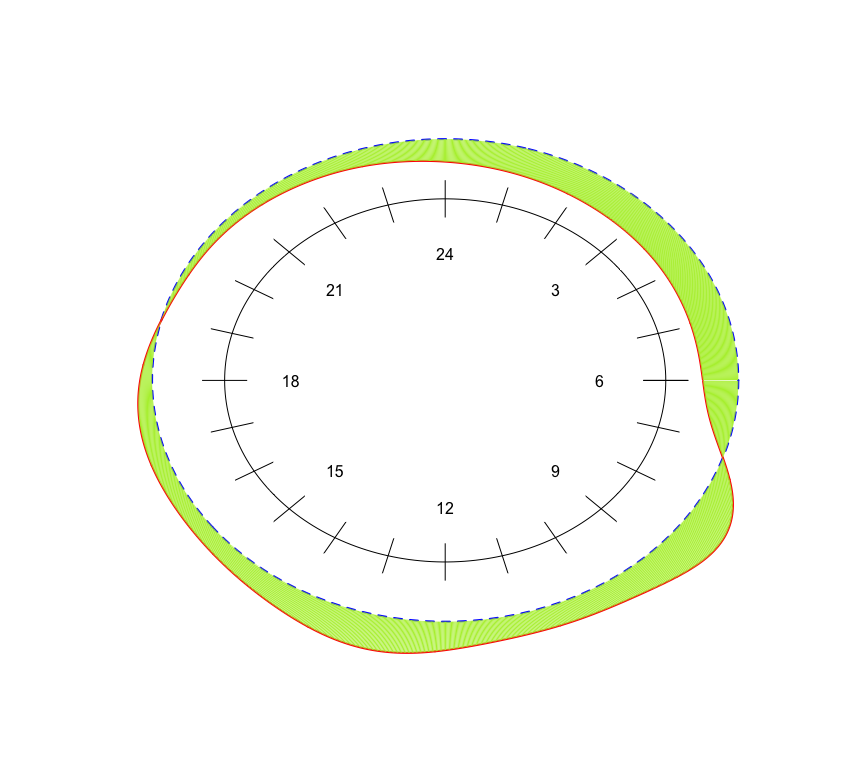}
	\caption{Estimated density of the times of birth in the US (2018) (red line) and its distance to a uniform distribution (blue dashed line) with data from \url{https://www.cdc.gov/nchs/data_access/vitalstatsonline.htm} plotted around a 24-hour clock face. Our test for uniformity is based on the size of the green area.}
	\label{img:graphic}
\end{figure} 

\paragraph*{Related literature.} 
Circular data, also called spherical, directional or wrapped (around the circumference of the unit circle), appears in various applications. For an in-depth review of many examples for circular data we refer the reader to \cite{Mardia1972}, \cite{Fisher1995} and \cite{MardiaJupp2009}. Let us only briefly mention two popular fields of application. Circular models are used for data with a temporal or periodic structure, where the circle is identified e.g. with a clock face (cp. \cite{GillHangartner2010}). Moreover, identifying the circle with a compass rose, directional data can also be represented by a circular model. \cite{KerkyacharianNgocPicard2011a} and \cite{LacourNgoc2014}, for instance, investigated a circular model with multiplicative error. Nonparametric estimation in the additive error model \eqref{model} has amongst others been considered in \cite{Efromovich1997}, \cite{ComteTaupin2003} and \cite{JohannesSchwarz2013a}.

Quadratic functional estimation in direct models has received a lot of attention in the literature, let us only mention a few references.  \cite{BickelRitov1988} and \cite{BirgeMassart1995} establish minimax rates for the estimation of functionals of a density, where they discovered a typical phenomenon in quadratic functional estimation: the so-called elbow effect, which also appears in our results. It refers to a sudden change in the behaviour of the rates, as soon as the smoothness parameter crosses a critical threshold. 

In a Gaussian sequence space model, which is closely related to our model, for instance,  \cite{LaurentMassart2000},  \cite{Laurent2005}  consider adaptive quadratic functional estimation via model selection, \cite{CaiLow2005} and \cite{CaiLow2006} derive minimax optimal estimators under Besov-type regularity assumptions. \cite{CollierCommingesTsybakov2017} consider sparsity constraints. 
Quadratic functional estimation in an inverse Gaussian sequence space model is treated by \cite{ButuceaMeziani2011} (known operator) and \cite{Kroll2019} (partially unknown operator). For quadratic functional estimation for deconvolution on the real line we refer to
\cite{Butucea2007} and \cite{Chesneau2011}.

Concerning the testing task, in the literature there exist several
definitions of rates and radii of testing in an asymptotic and
nonasymptotic sense. The classical definition of an asymptotic rate
of testing for nonparametric alternatives is essentially introduced
in the series of papers  \cite{Ingster1993},
\cite{Ingster1993a} and \cite{Ingster1993b}. For fixed noise levels, two alternative definitions of a nonasymptotic radius of testing are typically considered. For prescribed error
probabilities $\alpha,\beta\in(0,1)$, \cite{Baraud2002}, \cite{LaurentLoubesMarteau2012} and
\cite{MarteauSapatinas2017}, amongst others, define a nonasymptotic radius of testing as the smallest separation radius
$\rho$ such that there is an $\alpha$-test with maximal type II error
probability over the $\rho$-separated alternative smaller than $\beta$.  The definition we us in this paper -- which is based on the sum of both error probabilities -- is adapted e.g. from \cite{CollierCommingesTsybakov2017}. 
The connection between quadratic functional estimation and testing has for example been used in \cite{CollierCommingesTsybakov2017} (in a direct Gaussian sequence space model under sparsity), \cite{Kroll2019} (in an indirect Gaussian sequence sequence space model under regularity constraints) and \cite{Butucea2007} (in a convolution model on the real line). Let us now introduce our setting.

\paragraph*{Quadratic functional estimation.} Denote by $\densities$ the subset of real probability densities in $\Lp[2] := \Lp[2]([0,1))$, the Hilbert space of square-integrable complex-valued functions on $[0,1)$ equipped with its usual norm $\lV \cdot \rV_{\Lp[2]}$. Since we are interested in the estimation of the quadratic functional $\tqF(\xden) = \lV \xden \rV_{\Lp[2]}^2$ of $\xden$, we assume throughout this paper that both $\xden$ and $\epsden$ (and, hence, $\yden$) belong to $\densities$. We also want to compare $f$ to the density $\xoden = \mathds{1}_{[0,1]}$ of a uniform distribution by estimating their $\Lp[2]([0,1))$-distance $\qF(\xden) = \lV \xden  - \xoden \rV_{\Lp[2]}$. Since $	\qF(\xden) = \tqF(\xden) - 1$, these problems are equivalent and we will focus on the estimation of $\qF(\xden)$.  Let $\lcb \yOb \rcb_{k \in \lcb1, \dots, n \rcb}$ be $\say$ independent and identically distributed observations with density $\yden$, i.e. the observations are given by
\begin{align}
	\label{observations}
	\yOb \iid \yden = \xden \ccon \epsden, \qquad \nb \in \set{1, \dots, \say}.
\end{align}
Denote by $\Pf[\xden]$ and $\Ef[\xden]$ the probability distribution and the expectation associated with the data \eqref{observations}, respectively. 
For a nonparametric class of functions $\mc E$, we measure the accuracy of an estimator $\hqF$, i.e. a measurable function $\hqF: \IR^{\say} \to \IR$, by its \textit{maximal risk}
\begin{align*}
	\erisk[\hqF, \mc E] := \sup_{\xden \in \mc E} \Ef[\xden] \lb \hqF - \qF(\xden) \rb^2
\end{align*}
and compare its performance to the \textit{minimax risk of estimation}
\begin{align*}
	 \erisk[\mc E] := \inf_{\hqF} \erisk[\hqF, \mc E],		
\end{align*}
where the infimum is taken over all possible estimators. An estimator $\hqF$ is called minimax optimal for the class $\mc E$, if its maximal risk is bounded by the minimax risk $\erisk[\mc E]$ up to a constant. 

\paragraph*{The testing task.} Based on the observations \eqref{observations}, we test the null hypothesis $\lcb \xden = \xoden \rcb$ against the alternative $\lcb \xden \neq \xoden \rcb$. To make the null hypothesis and the alternative distinguishable, we separate them in the $\Lp[2]$-norm. For a separation radius $\sera \in \pRz$, let us define $\Lp[2]_{\sera} := \lcb \xden \in \Lp[2] : \VnormLp{{\xden}} \geq \sera \rcb$, which is called the \textit{energy condition}. For a nonparametric class of densities $\mc E$, called the \textit{regularity condition}, the testing problem can be written as
\begin{align}
	\label{testing:problem}
H_0: \xden = \xoden \qquad \text{against} \qquad H_1^{\sera}: \xden  \in \Lp[2]_{\sera} \cap \mc E.
\end{align}
We measure the accuracy of a test $\test$, i.e. a measurable function $\test: \IR^n \to \lcb 0,1 \rcb$, by its maximal risk defined as the sum of the type I error probability and the maximal type II error probability over the $\sera$-separated alternative
\begin{align*}
	\trisk[\test \mid \mc E, \sera] := \Pf[\xoden](\test = 1) + \sup_{\substack{ \xden \in \Lp[2]_{\sera} \cap \mc E}} \Pf[\xden](\test = 0). 
\end{align*}
We are particularly interested in the smallest possible value of $\sera^2$ by which we need to separate the null and the alternative for them to be distinguishable. A value {$\sera^2(\test, \mc E)$ }is called \textit{radius of testing} of the test $\test$ over the alternative $\mc E$, if for all $\alpha \in (0,1)$ there exist constants $\uA_{\alpha}, \oA_{\alpha} \in \pRz$ such that
\begin{enumerate}
	\item[(i)] for all $A \geq \oA_{\alpha}$ we have $\trisk[\test \mid \mc E, A \sera(\test, \mc E)] \leq \alpha$, \hfill (upper bound)
	\item[(ii)] for all $A \leq \uA_{\alpha}$ we have $\trisk[\test \mid \mc E, A \sera(\test, \mc E)] \geq 1 - \alpha$. \hfill (lower bound)
\end{enumerate}
The difficulty of the testing problem can be characterized by the minimax risk
\begin{align*}
	\trisk[ \mc E, \sera] := 	\inf_{\test} \trisk[\test \mid \mc E, \sera] 
\end{align*}
where the infimum is taken over all possible tests. The value {$ \sera^2(\mc E)$} is called \textit{minimax radius of testing}, if for all $\alpha \in (0,1)$ there exist constants $\uA_{\alpha}, \oA_{\alpha} \in \pRz$ such that
\begin{enumerate}
	\item[(i)] for all $A \geq \oA_{\alpha}$ we have $\trisk[ \mc E, A \sera(\mc E) ] \leq \alpha$, \hfill (upper bound)
	\item[(ii)] for all $A \leq \uA_{\alpha}$ we have $\trisk[ \mc E, A \sera(\mc E) ] \geq 1 - \alpha$. \hfill (lower bound)
\end{enumerate}
If  $\sera^2(\mc E)$ is a radius of testing for the test $\test$, then the test is called minimax optimal.

\paragraph*{Methodology.} We characterise both the minimax risk and the minimax radius in terms of the sample size $\say$, the parameters of $\mc E$ and the error density $\epsden$.  Our approach heavily depends on the properties of the Hilbert space $\Lp[2]([0,1))$ equipped with its usual inner product $\langle \cdot , \cdot \rangle $. Given the exponential basis $\expb$, $j \in \IZ$ of $\Lp[2]([0,1))$, with $\expb(x) = \exp(2 \pi i \expan x)$ for $x \in [0,1)$, we denote the Fourier coefficients of $f \in \Lp[2]([0,1))$ by $f_j = \langle f, \expb \rangle$, $j \in \IZ$. This leads to the discrete Fourier series expansion 
\begin{align}
	\label{expansion}
	\xden = \sum_{\expan \in \IZ} \xden[\expan] \expb,
\end{align}
where equality holds in $\Lp[2]([0,1))$. The nonparametric class of functions $\mc E$ is formulated in terms of the Fourier coefficients and characterises the regularity of the function. Let $R > 0$ and let $\wdclass[] = \Nsuite{\wdclass}$ be a strictly positive, monotonically non-increasing sequence. We assume that the density of interest $\xden$  belongs to the $\Lp[2]$-ellipsoid
\begin{align}
	\rwdclass = \lcb \xden \in \densities : 2 \sum_{\expan \in \IN  }  \wdclass^{-2} \lv \xden[\expan] \rv^2 \leq \rdclass^2 \rcb. \label{ellipsoid} 
\end{align}
Note that $\xden \in \rwdclass$ imposes conditions on all coefficients $\xden[\expan]$, $j \in \IZ$, since $\lv f_j \rv^2 =\lv {f_{-j}} \rv^2$, $j \in \IN$, for all real-valued functions and, additionally, $f_0 = 1$ for all densities. 
The definition \eqref{ellipsoid} is general enough to cover classes of ordinary and analytically smooth densities. Expanding $\xden$ and $\xoden$ in the exponential basis as in \eqref{expansion} and applying Parseval's Theorem yields a representation of the quadratic functional $\qF(\xden) = \lV \xden - \xoden \rV_{\Lp[2]}^2$ in their Fourier coefficients
$	\qF(\xden) = \sum_{\expan \in \IZ} \lv \xden[\expan] - \xoden[\expan]\rv^2 = 2 \sum_{\expan \in \IN} \lv \xden[\expan] \rv^2$. 
Moreover, by the circular convolution theorem we have $g = f \ccon \epsden$ if and only if the Fourier coefficients satisfy $\yden[\expan] = \xden[\expan] \cdot \epsden[\expan]$ for all $\expan \in \IZ$.  Here and subsequently, we assume that the Fourier coefficients of the noise density $\epsden$ are non-vanishing everywhere, i.e. $\lv \epsden[\expan] \rv > 0$ for all $\expan \in \IZ$. The quadratic functional can then be expressed as 
\begin{align}
	\label{qf}
	\qF(\xden) = \sum_{\expan \in \IZ} \frac{ \lv \yden[\expan] - \epsden[\expan] \xoden[\expan]\rv^2 }{\lv \epsden[\expan] \rv^2} = 2 \sum_{\expan \in \IN} \frac{\lv \yden[\expan] \rv^2}{\lv \epsden[\expan] \rv^2}.
\end{align}
The only unknown quantities in \eqref{qf} are the Fourier coefficients $\yden[\expan]$, $\expan \in \IZ$, of $\yden$, which can easily be estimated. Since for $\expan \in \IZ$, $\yden[\expan] =\langle \yden, \expb \rangle = \Ef[\xden] (\expb(- \yOb[1]))$, a natural estimator is given by replacing the expectation with the empirical counterpart
$\hyden[\expan] = \frac{1}{\say} \sum_{\nb = 1}^{\say} \expb(-\yOb)$.
Inserting these estimators into the quadratic functional, however, generates a bias in every component. Since $\lv\hyden[\expan] \rv^2 - \frac{1 - \lv \hyden[\expan] \rv^2}{n-1}$ is an unbiased estimator of the numerator $\lv \yden[\expan] \rv^2 $, for $j \in \IN$, for each $k \in \IN$ we consider the estimator 
\begin{align}
	\label{estimator}
	\hqF_{k} := 2 \sum_{ j = 1}^n  \lv \epsden[\expan] \rv^{-2}  \lcb \lv\hyden[\expan] \rv^2 - \frac{1 - \lv \hyden[\expan] \rv^2}{n-1} \rcb,
\end{align}
which is an unbiased estimator of the truncated quadratic functional $\qF_k := 2 \sum\limits_{j=1}^n \frac{ \lv \yden[\expan] \rv^2 }{\lv \epsden[\expan] \rv^2}$. 

 Using $\hqF_k$ (defined in \eqref{estimator}) as an estimator of the distance $\lV \xden - \xoden \rV_{\Lp[2]}^2$ to the null hypothesis, we construct a test that, roughly speaking, compares the estimator to a multiple of its standard deviation.
 Precisely, for $k \in \IN$ and a constant $C_\alpha$, we consider the test 
 \begin{align}
 	\label{test}
 	\test_{\alpha,k} := \Ind{\lcb \hqF_k \geq C_\alpha \nu_k^2 \rcb } \qquad \text{ with } \qquad \nu_k^2 := \frac{1}{\say} \sqrt{\sum_{0 < \lv j \rv \leq k} \frac{1}{\lv \epsden[j] \rv^4} }.
 \end{align}

 \paragraph*{Minimax results.}{We show that for fixed $k$ the risk of the estimator $\hqF_k$ in \eqref{estimator} is bounded by \begin{align}
 		\label{baselevel}
 \sera_k^4 \vee \baseterm \quad \text{with} \quad \sera_k^4 :=	\lcb \wdclass[k]^4 \vee \nu_k^4 \rcb\quad  \text{ and } \quad   \baseterm := \max_{m \in \IN}  \lcb \wdclass[m]^4  \wedge \frac{\wdclass[m]^{2} }{n \lv \epsden[m] \rv^{2} } \rcb
 \end{align}
up to a constant.
The base level term $\baseterm$ is present for all $k \in \IN$, whereas the term $\sera_k^4$, which represents a typical bias-variance trade-off, explicitly depends on the dimension parameter $k \in \IN$ and can, thus, be optimised with respect to $k$. More precisely, choosing $\optdimest$ as a minimizer of $\sera_k^4$, the risk of $\hqF_{\optdimest}$ is up to a constant bounded by
\begin{align}
	\label{upper:bound:est}
	\msera^4 \vee \baseterm := \lcb \min_{k \in \IN} \sera_k^4 \rcb    \vee  \baseterm. 
\end{align}
The term $\baseterm$ causes the classical elbow effect in quadratic functional estimation, since it prevents the rate from being faster than parametric. The upper bound shows the expected behaviour: a faster decay of the Fourier coefficients of $\epsden$, i.e. a smoother error density, results in a slower rate. Therefore, we call the decay of $\lcb \lv \epsden[\expan] \rv \rcb_{\expan \in \IN}$ the \textit{degree of ill-posedness} of the model \eqref{model}.  On the other hand, a faster decay of the Fourier coefficients of the density of interest $\xden$ yields a faster rate. We use the estimation upper bound to determine an upper bound for a radius of testing of the test $\test_{\alpha,k}$ defined in \eqref{test}. For appropriately chosen $C_\alpha$, an upper bound for the radius of testing of $\test_{\alpha,k}$ is given by
\begin{align*}
	\sera_k^2 =  \wdclass[k]^2 \vee  \frac{1}{\say} \sqrt{\sum_{0 <\lv \expan \rv \leq k} \tfrac{1}{\lv \epsden[\expan] \rv^{4}}} , 
\end{align*}
which can again be optimised with respect to $k \in \IN$. Again choosing $\optdim$ as the minimizer of $\sera_k^2$, the radius of testing of $\test_{\alpha,\optdim}$ is of order
\begin{align}
	\label{upper:bound:test}
	\msera^2 = \min_{k \in \IN}	\lcb   \wdclass[k]^2 \vee  \frac{1}{\say} \sqrt{\sum_{0 < \lv \expan \rv \leq k} \tfrac{1}{\lv \epsden[\expan] \rv^{4}}}  \rcb.
\end{align}
 Interestingly, the term causing the elbow effect in the estimation case vanishes in the radius of testing. Roughly speaking, the densities that cause $\baseterm$ in \eqref{upper:bound:est} and, hence, the elbow effect, are difficult to estimate (since they have large energy), but easy to test (since they are far from the null). This observation is explicitly used in the proof of the testing upper bound.
}	

 \paragraph*{Outline of the paper.} The upper bound for the estimation risk and the radius of testing is derived in \cref{sec:upper:est} and \cref{sec:upper:test}, respectively. \cref{sec:lower:test} provides a matching lower bound for the testing problem. In \cref{sec:testing:vs:estimation} we first show that testing is faster than quadratic functional estimation if we correct for the missing square, formally $r^4(\mc E) \geq C \sera^2(\mc E)$ for some $C > 0$. Using this connection between quadratic functional estimation and testing, we immediately obtain a lower bound for the estimation problem. It remains to prove an additional lower bound for the term $\baseterm$ in \eqref{upper:bound:est} that causes the elbow effect. Thus, we establish the order of both the minimax estimation rate and the minimax radius of testing. Technical results and their proofs are deferred to \cref{appendix}.


%
 \section{Upper bound for the estimation risk}	
 The next proposition presents an upper bound for the quadratic functional estimator defined in \eqref{estimator} for arbitrary $\xden \in \densities$ and $k \in \IN$. The key element of the proof is rewriting the estimator as a U-statistic and exploiting a well-known formula for its variance. 
 \label{sec:upper:est}
 	\begin{proposition}[Upper bound for the estimation risk]
 		\label{prop:ub}
 	 			For $n \geq 2$ and $k \in \IN$ the estimator defined in \eqref{estimator} satisfies
 	 			\begin{align}
 	 		 	 &	\Ef[\xden] \lb \hqF_k - \qF(\xden) \rb^2  \leq \lb \sum_{\lv \expan \rv > k} \lv \xden[\expan] \rv^2 \rb^2 +   \frac{c}{\say^2} \sum_{0 < \lv \expan \rv \leq k} \frac{1}{\lv \epsden[\expan]\rv^{4}} + \frac{c}{\say} \sum_{0 <  \lv \expan \rv \leq k} \frac{\lv \xden[\expan] \rv^2}{\lv \epsden[\expan] \rv^{2}}   \label{bias:variance:upper:bound}
 	 			\end{align}
  				with $c := \VnormInf{ \xden \ccon \epsden } = \VnormInf{\yden} := \sup_{x \in [0,1]} \lv \yden(x) \rv$.
 	 		\end{proposition}
  			\begin{proof}[Proof of \cref{prop:ub}] The bound follows from a classical bias-variance decomposition of the risk;
  				\begin{align}
  					\label{bias}
  						\Ef[\xden] \lb \hqF_k - \qF(\xden) \rb^2  = \lb \sum_{\lv \expan \rv > k} \lv \xden[\expan] \rv^2 \rb^2  + \Vf[\xden]\lb \hqF_k \rb.
  				\end{align}
  			To bound the variance, we rewrite the estimator as a U-statistic
  			\begin{align*}
  				\hqF_k & =  \frac{1}{n(n-1)} \sum_{l \neq m }  \sum_{0 < \lv j \rv \leq k}  \frac{  \expb(-\yOb[l]) \expb(\yOb[m])}{ \lv \epsden[j] \rv^2} =:  \frac{1}{n(n-1)} \sum_{l \neq m } h(Y_l, Y_m) =: \frac{1}{2} U_n,
  			\end{align*}	
  			where $h(y_1, y_2) :=  \sum_{0 < \lv j \rv \leq k}  \frac{  \expb(-y_1) \expb(y_2)}{ \lv \epsden[j] \rv^2}$ for $y_1, y_2 \in [0,1)$ and $U_n := \binom{n}{2}^{-1} \sum_{l \neq m } h(Y_l, Y_m)$. The kernel $h$ is symmetric and real-valued, i.e. $h(y_2,y_1) =  h(y_1, y_2)$ equals its complex conjugate $\overline{h(y_1, y_2)}$. Let us define the function $h_1: [0,1) \lra \IC, y \mapsto h_1(y) := \Ef(h(y,Y_2))$. 
  			By Lemma A on p. 183 in \cite{Serfling2009}, the variance of the U-statistic $U_n$ is determined by 
  			\begin{align*}
  				\Vf(U_n) = \binom{n}{2}^{-1} \lb 2(n-2) \xi_1 + \xi_2 \rb \quad \text{with } \xi_1 := \Vf(h_1(Y_1)) \text{ and } \xi_2 := \Vf(h(Y_1, Y_2)).
  			\end{align*}
  			Next, we bound the two terms $\xi_1$ and $\xi_2$. Since $h_1(y) = \Ef(h(y,Y_2)) = \sum_{0 < \lv \expan \rv \leq k} \frac{\expb(-y)}{\epsden[\expan]}  \frac{\Ef \expb(Y_2)}{\overline{\epsden[\expan]}} =  \sum_{0 < \lv \expan \rv \leq k} \frac{\overline{\yden[\expan]}}{\lv \epsden[\expan] \rv^2}\expb(-y) $, we obtain by Parseval's identity
  			 	\begin{align*}
  				\xi_1 \leq \Ef \lv h_1(Y_1) \rv^2 \leq \lV \yden \rV_\infty \lV h_1 \rV_{\Lp[2]}^2 = \lV \yden \rV_\infty \sum_{0 < \lv \expan \rv \leq k} \frac{\lv \xden[\expan] \rv^2}{\lv \epsden[\expan] \rv^2}.
  			\end{align*}
  			Now consider the term $\xi_2$. It holds 
  			\begin{align*}
  				\xi_2 = \Vf(h(Y_1, Y_2)) \leq \Ef \lv h(Y_1, Y_2) \rv^2 \leq \lV \yden \rV_\infty \int \int \lv h(y_1, y_2) \rv^2 \dif y_1 g(y_2) \dif y_2
  			\end{align*} 
  			where 
  			\begin{align*}
  				\int \lv h(y_1, y_2) \rv^2 \dif y_1 = \sum_{ 0 < \lv \expan \rv, \lv l \rv \leq k} \frac{1}{\lv \epsden[\expan] \rv^2 \lv \epsden[l] \rv^2} \int_0^1 \expb(y_2 - y_1) \overline{\expb[l](y_2 - y_1)} \dif y_1  = \sum_{ 0 < \lv \expan \rv \leq k} \frac{1}{\lv \epsden[\expan] \rv^4} 
  			\end{align*}
  			and, hence,
  			\begin{align*}
  				\int \int \lv h(y_1, y_2) \rv^2 \dif y_1 g(y_2) \dif y_2 = \sum_{ 0 < \lv \expan \rv \leq k} \frac{1}{\lv \epsden[\expan] \rv^4}  \int g(y_2) \dif y_2 =  \sum_{ 0 < \lv \expan \rv \leq k} \frac{1}{\lv \epsden[\expan] \rv^4} 
  			\end{align*}
  			Finally, combining the bounds for $\xi_1$ and $\xi_2$ yields
  			\begin{align}
  				\label{variance:bound}
  					\Vf(\hqF_k) = \frac{1}{4} \Vf(U_n) =  \frac{ 2(n-2) \xi_1 + \xi_2}{2 n(n-1)}  \leq \frac{\VnormInf{\yden} }{n} \sum_{0 < \lv \expan \rv \leq k} \frac{\lv \xden[\expan] \rv^2}{\lv \epsden[\expan] \rv^2} + \frac{\VnormInf{\yden} }{n^2}  \sum_{ 0 < \lv \expan \rv \leq k} \frac{1}{\lv \epsden[\expan] \rv^4} 
  			\end{align}
  		where we used that $\tfrac{1}{2(n-1)} \leq \tfrac{1}{n}$ for $n \geq 2$. Together with \eqref{bias}, this proves the assertion. 
   \end{proof}
 	The upper bound in \eqref{bias:variance:upper:bound} depends on the quantity $c =  \VnormInf{\yden} \leq  \VnormInf{\epsden}$, which is uniformly bounded for all $f \in \densities$ if  $\ \VnormInf{\epsden} < \infty$. By additionally exploiting the regularity condition \eqref{ellipsoid}, we obtain a uniform bound for the risk, valid for all $\xden \in \rwdclass$. 
  				\begin{corollary}[Uniform upper bound for the risk of estimation]
  					\label{prop:ub:uni}
	  			Consider $\nu_k^4 $ and $\baseterm$ as defined in \eqref{test} and \eqref{baselevel}, respectively. For $n, k \in \IN, n \geq 2$ the estimator defined in \eqref{estimator} satisfies 
  				{ 	\begin{align}
  					\label{re:ub:est}
  					\sup_{\xden \in \rwdclass}\Ef[\xden] \lb \hqF_k - \qF(\xden) \rb^2  \leq \quad  c_1 \wdclass[k]^4 \quad  \vee \quad  c_2 \nu_k^4  \quad \vee  \quad c_3 \baseterm
  				\end{align}}
  				with $c_1 := 3  \rdclass^4 $,  $c_2  :=3(\VnormInf{\epsden}+\rdclass^2)$, $c_3 := 3   \VnormInf{\epsden} \rdclass^2$.
  			\end{corollary}
  			\begin{proof}[Proof of \cref{prop:ub:uni}]
  				We exploit the upper bound in \eqref{bias:variance:upper:bound}.
  				Since the sequence $\wdclass[]$ is non-increasing, the first term on the right-hand side in \eqref{bias:variance:upper:bound} (the bias term) is bounded by
  				\begin{align*}
  					\sum_{\lv \expan \rv > k} \lv \xden[\expan] \rv^2 = \sum_{\lv \expan \rv > k} \lv \xden[\expan] \rv^2 \wdclass[\lv \expan \rv]^{-2} \wdclass[\lv \expan \rv]^{2} \leq  \sum_{\lv \expan \rv > k} \lv \xden[\expan] \rv^2 \wdclass[\lv \expan \rv ]^{-2} \wdclass[k]^{2} \leq \rdclass^2 \wdclass[k]^2.
  				\end{align*}
  				To bound the second term on the right-hand side of \eqref{bias:variance:upper:bound}, we bound each summand, i.e. for each $j \in \IN$ we have $	\frac{1}{\say}	 \frac{ \lv \xden[\expan] \rv^2}{\lv \epsden[\expan] \rv^{2}} \leq  \frac{\lv \xden[\expan] \rv^2}{\wdclass[\expan]^2} \wdclass[\expan]^4 \lb 1 \wedge  \frac{1}{\say \lv \epsden[\expan] \rv^{2} \wdclass[\expan]^{2}} \rb$ if $\say \lv \epsden[\expan] \rv^{2} \wdclass[\expan]^{2} \geq 1$ and $	\frac{1}{\say}	 \frac{ \lv \xden[\expan] \rv^2}{\lv \epsden[\expan] \rv^{2}} \leq \frac{\rdclass^2}{\say^2 \lv \epsden[\expan]\rv^4} $ otherwise.
  				Hence, we obtain a bound for the entire sum
  				\begin{align*}
  				\frac{1}{n}	\sum_{0 < \lv \expan \rv \leq k} \frac{ \lv \xden[\expan] \rv^2}{\lv \epsden[\expan] \rv^{2}} & \leq \sum_{0 < \lv \expan \rv \leq k} \frac{\lv \xden[\expan] \rv^2}{\wdclass[\expan]^2} \wdclass[\expan]^4 \lb 1 \wedge  \frac{1}{\say \lv \epsden[\expan] \rv^{2} \wdclass[\expan]^{2}}\rb + \frac{\rdclass^2}{\say^2} \sum_{0 < \lv \expan \rv \leq k} \frac{1}{\lv \epsden[\expan]\rv^4}  \leq R^2 \baseterm + R^2 \nu_k^4.
  				\end{align*}  	
  				Combining both bounds yields the assertion.			
  			\end{proof}
  	
			\begin{remark}[Optimal choice of the dimension parameter]
				The first two terms in the upper bound \eqref{re:ub:est} depend on the dimension parameter $k \in \IN$, whereas the last term $c_3 \baseterm$ does not. It plays the role of a base-level error, which causes the well-known elbow effect in quadratic functional estimation (cp. also \cref{ill:qf} below). It can easily be seen that $\baseterm$ is always of order larger than $\frac{1}{n}$.  In other words, no matter the choice of $k$ the estimation rate can never be faster than parametric. The first two terms, however, depend on $k \in \IN$ and can therefore be optimised. We define the optimal dimension
								\begin{align}
										\label{optdimest}
											\optdimest =\min \lcb k \in \IN :   \wdclass[k]^4 \leq   \frac{1}{\say^2} \sum_{ 0 < \lv \expan \rv \leq k} \frac{1}{\lv \epsden[\expan]\rv^{4}}   \rcb.
								\end{align}
				as the $k$ that achieves an optimal bias-variance trade-off. 
			\end{remark}  	
		
			\begin{theorem}[Upper bound for the minimax risk of estimation]
				\label{thm:ub}
				For $n \geq 2$ and $\optdimest$ as in \eqref{optdimest}
			{ 	\begin{align}
					\label{opt:ub}
				r^2(\rwdclass)   \leq \erisk[\hqF_{\optdimest},\rwdclass] \leq  C \lb \msera^4 \vee \baseterm \rb
				\end{align}
			with $C := \rdclass^4 + \VnormInf{\epsden}+\rdclass^2 +    \VnormInf{\epsden} \rdclass^2$.
			}
			\end{theorem}
			\begin{proof}[Proof of \cref{thm:ub}]
				We apply \cref{prop:ub:uni} to $\hqF_{\optdimest}$ with $\optdimest$ as in \eqref{optdimest}. 
			\end{proof}
			We now  provide an additional upper bound for the variance of the estimator \eqref{estimator}, which is used in the next section to derive an upper bound for the testing radius.
  	
  			\begin{corollary}[Upper bound for the variance]
  				\label{upper:bound:variance}
  				Let $\xoden = \mathds{1}_{[0,1]}$ and $f \in \densities$. For $n, k \in \IN, n \geq 2$ and  $\nu_k^2$ as in \eqref{test} the estimator defined in \eqref{estimator} satisfies	
  				\begin{align}
  						\Vf[\xoden] (\hqF_k) & \leq \nu_k^4, \label{eq:variance:upper:null} \\
  					\label{eq:variance:upper}
  					\Vf (\hqF_k) & \leq  \VnormInf{\epsden} \cdot \qF_k(\xden) \nu_k^2 + \VnormInf{\epsden}  \cdot \nu_k^4.   					
  				\end{align}
   			\end{corollary}
  			\begin{proof}[Proof of \cref{upper:bound:variance}]
  					We use the bound  \eqref{variance:bound} derived in the proof of \cref{prop:ub}.
  				The first term on the right hand side can be bounded due to the Cauchy-Schwarz inequality by
  				\begin{align*}
  						\sum_{0 < \lv \expan \rv \leq k} \frac{\lv \xden[\expan] \rv^2}{\lv \epsden[\expan] \rv^{2}} \leq \sqrt{	\sum_{0 < \lv \expan \rv \leq k} {\lv \xden[\expan] \rv^4}} \sqrt{	\sum_{0 < \lv \expan \rv \leq k} \frac{1}{\lv \epsden[\expan] \rv^{4}} } \leq 	\sum_{0 < \lv \expan \rv \leq k} {\lv \xden[\expan] \rv^2} \sqrt{	\sum_{0 < \lv \expan \rv \leq k} \frac{1}{\lv \epsden[\expan] \rv^{4}} } =   \qF_k(\xden) n  \nu_k^2,
   				\end{align*}
   				exploiting $\sqrt{x+y} \leq \sqrt{x} + \sqrt{y}$ for any $x,y \geq 0$  in the last inequality. Combining this bound with  $\VnormInf{\yden} \leq \VnormInf{\epsden}$ shows \eqref{eq:variance:upper}. Additionally, for $\xden = \xoden = \mathds{1}_{[0,1]}$, and hence $g = \mathds{1}_{[0,1]}$, we have $\VnormInf{\yden} = 1$ and $\qF_k(\xden) = 0$, which proves \eqref{eq:variance:upper:null}.
  			\end{proof}
  		
  		\begin{illustration}
  			\label{ill:qf}
  			 Throughout the paper we illustrate the order of the estimation risk under the following typical
  			smoothness and ill-posedness assumptions for the density of interest $\xden$ and the noise density $\epsden$, respectively. 	For two real-valued sequences $(x_n)_{n \in \IN}$ and $(y_n)_{n \in \IN}$ we write $x_n \lesssim y_n$ if there exists a constant $c > 0$ such that for all $n \in \IN$, $x_n \leq c y_n$. We write $x_n \sim y_n$, if both $x_n \lesssim y_n$ and $y_n \lesssim x_n$. We call $y_n$ the \textbf{order} of $x_n$.  Concerning the class
  			$\rwdclass$ we distinguish two behaviours of the sequence $\wdclass[]$,
  			namely the \textbf{ordinary smooth} case
  			$\wdclass \sim { j^{-\sPara}}$ for $\sPara > 1/2$ where $\rwdclass$
  			corresponds to a Sobolev ellipsoid, and the \textbf{super smooth}
  			case $\wdclass \sim {\exp(-j^{\sPara})}$ for $\sPara > 0$, corresponding to a class of analytic functions. 
			We also distinguish two cases for the regularity of the error density
  			$\epsden[]$. For $\pPara>1/2$ we consider a \textbf{mildly
				ill-posed} model $\lv \epsden[j] \rv \sim \lv j \rv ^{-\pPara}$ and for $p > 0$ a
  			\textbf{severely ill-posed} model
  			$ \lv \epsden[j] \rv  \sim {\exp(-\lv j \rv^{\pPara})}$. 
Many examples of circular densities can be found in Chapter 3 of \cite{MardiaJupp2009}.  The table below presents the order of the upper bound for $\mrate^2$ in \eqref{opt:ub}, in \cref{sec:testing:vs:estimation} we provide a matching lower bound, thus establishing the rate-optimality of the upper bound. The derivations of the risk bounds can be found in \cref{calculations}. \\ [3ex]
  			\centerline{\begin{tabular}{ll|ll|l}\toprule
  					\multicolumn{5}{c}{Order of the minimax estimation risk $\mrate^2 = \msera^4 \vee \baseterm$}\\\midrule
  					${\wdclass[j]}$ & ${\lv \epsden[ j ] \rv}$ & $\msera^4$ & $\baseterm$& $\mrate^2$  \\
  							(smoothness) & (ill-posedness) & & & \\
  					\midrule
  					${ j^{-\sPara}}$ & ${\lv j \rv^{-\pPara}}$ & $\say^{-\tfrac{8\sPara}{4\sPara + 4 \pPara + 1}}$ & $\begin{cases} \say^{-\frac{8 \sPara}{4 \sPara + 4 \pPara}} & s - p < 0 \\\say^{-1} & s - p \geq 0 \end{cases}$ & $ \begin{cases} \say^{-\tfrac{8\sPara}{4\sPara + 4 \pPara + 1}} & \sPara - \pPara < \tfrac{1}{4} \\
  					\say^{-1}  & \sPara - \pPara \geq  \tfrac{1}{4}\end{cases}$ \\
  					${ j^{-\sPara}}$ & 	${e^{-\lv j \rv^{\pPara}}}$ & $(\log \say)^{-\tfrac{4\sPara}{\pPara}}$ &  $(\log\say)^{-\tfrac{4\sPara}{\pPara}}$ & $(\log\say)^{-\tfrac{4\sPara}{\pPara}}$ \\
  					${e^{-j^{\sPara}}}$ & ${\lv j \rv^{-\pPara}}$& $\say^{-2} (\log n)^{\frac{4 \pPara}{\sPara}}$ & $\say^{-1}$ & $\say^{-1}$ \\
  					\bottomrule
 				\end{tabular}}
  		\end{illustration}

\section{Upper bound for the radius of testing}
\label{sec:upper:test}
In this section we derive an upper bound for the radius of testing of the task \eqref{testing:problem}. We consider the test	$\test_{\alpha,k}= \Ind{\lcb \hqF_k \geq C_\alpha \nu_k^2 \rcb }$ defined in \eqref{test}, that is based on the estimator $\hqF_k$ in \eqref{estimator} of the distance  $\lV \xoden - \xden \rV_{\Lp[2]}^2$ to the null hypothesis. 

	\begin{proposition}[Upper bound for the radius of testing of $\test_{\alpha,k}$] 
	\label{r:upper:bound:testing}
	Let $\alpha \in (0,1)$, $c := \lV \epsden \rV_\infty$ and $C_\alpha, \tA_{\alpha} \in \pRz$ be such that
	\begin{align}
		\label{choice:of:A}
		\frac{2 C_\alpha + 1}{C_\alpha^2} c \leq \frac{\alpha}{2} \qquad \text{and } \qquad 
		&  		
		\frac{2 C_\alpha + 1}{\lb \tA_\alpha - C_\alpha \rb^2} c \leq \frac{\alpha}{2}.
	\end{align}
	 Set $\oA_{\alpha}^2 := R^2 + \tA_{\alpha}^2$. Then, for all $A \geq \oA_{\alpha}$ and all $k \in \IN$ we obtain
	\begin{align}
		\label{eq:upper:bound:testing}
		\trisk[ \test_{\alpha, k} \mid \rwdclass, A \rho_k] \leq \alpha,
	\end{align}
	i.e. $\rho_k^2$	is an upper bound for the radius of testing of $\test_{\alpha,k}$. 
\end{proposition}
	\begin{remark}[Choice of $C$ and $\oA_\alpha$]  
 In particular, \eqref{choice:of:A}, and, hence, \cref{r:upper:bound:testing} is satisfied for 
		$C_\alpha = 6 \alpha^{-1} \VnormInf{\epsden}$ and $\tA_\alpha = C_\alpha +2 \alpha^{-1} \sqrt{12 \VnormInf{\epsden}^2 \alpha^{-1} + \VnormInf{\epsden}  }  $.
\end{remark}
\begin{proof}[Proof of \cref{r:upper:bound:testing}]
	We show that both the type I error probability and the type II error probability of the test \eqref{test} are bounded. Consider first the\textbf{ type I error probability}. Applying first Markov's inequality and then the second inequality \eqref{eq:variance:upper:null} from \cref{upper:bound:variance}, we obtain
	\begin{align}
		\label{underthenull}
		\Pf[\xoden](\test_{\alpha,k} = 1) = \Pf[\xoden](\hqF_k \geq C_\alpha \nu_k^2 ) \leq \frac{\Ef[\xoden] \lb \lb \hqF_k \rb^2 \rb}{C_\alpha^2 \nu_k^4} =  \frac{\Vf[\xoden] \lb  \hqF_k  \rb}{C_\alpha^2 \nu_k^4} \leq \frac{1}{C_\alpha^2} \leq \frac{\alpha}{2},
	\end{align}
	for all $C_\alpha$ satisfying \eqref{choice:of:A}, since $\lV \epsden \rV_\infty \geq 1$. Next, we consider the \textbf{type II error probability}. Let $\xden$ be contained in the $(\oA_\alpha  \rho_k)$-separated alternative, i.e. $\xden \in \rwdclass$ and $\qF(\xden) \geq \lb \oA_\alpha \rb^2 \rho_k^2$. We expand 
	\begin{align*}
		\Pf[\xden](\test_{\alpha,k} =0) = \Pf[\xden](\hqF_{k} < C_\alpha \nu_k^2 ) = \Pf[\xden]\lb \hqF_{k} - \qF_{k}(\xden) < C_\alpha \nu_k^2 - \qF_{k}(\xden) \rb
	\end{align*}
	and distinguish the following two cases for the density $f$
	\begin{enumerate}
		\item  $\qF_{k}(\xden) \geq 2 C_\alpha \nu_k^2$, \hfill (easy to test)
		\item $\qF_{k}(\xden) < 2 C_\alpha \nu_k^2$.  \hfill (difficult to test)
	\end{enumerate} 
	\textbf{Case 1. (easy to test)}  We have $C_\alpha \nu_k^2 - \qF_{k}(\xden) \leq - \tfrac{1}{2} \qF_{k}(\xden)$ and, therefore, due to Markov's inequality
	\begin{align*}
		\Pf[\xden](\test_{\alpha,k} = 0) & \leq \Pf[\xden](\hqF_{k} - \qF_{k}(\xden) \leq - \tfrac{1}{2} \qF_{k}(\xden)) =  \Pf[\xden]( \qF_{k}(\xden) - \hqF_{k}\geq \tfrac{1}{2} \qF_{k}(\xden)) 
		\leq 4 \frac{\Vf[\xden](\hqF_{k})}{\lb \qF_{k}(\xden) \rb^2}.
	\end{align*}
	On the one hand, by the case distinction, we have $\qF_{k}(\xden) \geq 2 C_\alpha \nu_k^2$, on the other hand we have $	\Vf (\hqF_{k}) \leq c \qF_{k}(\xden) \nu_{k}^2 + c \nu_{k}^4 $ with $c = \VnormInf{\epsden}$ due to \eqref{eq:variance:upper} in \cref{upper:bound:variance}. Hence,
	\begin{align*}
		\Pf[\xden](\test_{\alpha,k} = 0) & \leq 4 \frac{c \qF_{k}(\xden) \nu_{k}^2 + c \nu_{k}^4 }{\lb \qF_{k}(\xden)\rb^2 } = 4 \lb \frac{c \nu_{k}^2 }{\qF_{k}(\xden)} +  \frac{c \nu_{k}^4}{\lb \qF_{k}(\xden)\rb^2}  \rb\\
		&  \leq 4 \lb \frac{c \nu_{k}^2 }{2 C_\alpha \nu_k^2} +  \frac{c \nu_{k}^4}{4 C_\alpha^2 \nu_k^4}  \rb = \frac{2c}{C_\alpha} + \frac{c}{C_\alpha^2} \leq \frac{\alpha}{2}.
	\end{align*}
 	\textbf{Case 2. (difficult to test)} Under the alternative exploiting $\qF(\xden) = \sum_{0 < \lv \expan \rv < \infty } \lv \xden[\expan] \rv^2 \geq \lb \oA_\alpha\rb^2\rho_k^2$ and $\sum_{\lv \expan \rv > {k}} \lv \xden[\expan] \rv^2 \leq \wdclass[{k}]^2 \rdclass^2$,  it follows
\begin{align*}
	\qF_{k}(\xden) = \qF(\xden)  - \sum_{\lv \expan \rv > {k}} \lv \xden[\expan] \rv^2 \geq (\oA_\alpha)^2 \nu_k^2 - \wdclass[k]^2 \rdclass^2 = \tA_\alpha^2 \nu_k^2 + \wdclass[k]^2 \rdclass^2 - \wdclass[k]^2 \rdclass^2 = \tA_\alpha^2 \nu_k^2
\end{align*}
Hence, due to Markov's inequality, the type II error probability satisfies
\begin{align*}
	\Pf[\xden](\test_{\alpha,k} = 0) & = 	\Pf[\xden]\lb \hqF_{k} - \qF_{k}(\xden) \leq C_\alpha \nu_k^2 -  \qF_{k}(\xden) \rb
	\leq 	\Pf[\xden]\lb \hqF_{k} - \qF_{k}(\xden) \leq \lb C_\alpha - \tA_\alpha^2 \rb \nu_k^2 \rb \\
	& = \Pf[\xden]\lb - \hqF_{k} + \qF_{k}(\xden) \geq \lb - C_\alpha + \tA_\alpha^2  \rb \nu_k^2 \rb   \leq \frac{\Vf(\hqF_{k})}{ \lb \tA_\alpha^2  -  C_\alpha \rb^2 \nu_k^4}.
\end{align*}
By \eqref{eq:variance:upper:null} in \cref{upper:bound:variance}, the case distinction and the choice of $\tA_\alpha$ in \eqref{choice:of:A}, it follows
\begin{align*}
	\Pf[\xden](\test_{\alpha, k} = 0) & \leq  \frac{ c \qF_{k}(\xden) \nu_{k}^2 + c \nu_{k}^4 }{ \lb \tA_\alpha^2  - C_\alpha \rb^2 \nu_k^4} \leq \frac{2 c C_\alpha + c}{\lb \tA_\alpha^2  - C_\alpha \rb^2} \leq \frac{\alpha}{2}.
\end{align*}
Combining the last bound and \eqref{underthenull}, we obtain the assertion, which completes the proof.
\end{proof}

From \cref{r:upper:bound:testing} with $\optdim$ as in \eqref{optdimest} and $\msera$ as in \eqref{upper:bound:test}, we immediately obtain the following corollary and, hence, omit the proof.

\begin{corollary}[Upper bound for the minimax radius of testing]
	\label{ub:msera}
		Under the conditions of \cref{r:upper:bound:testing} for all $A \geq \oA_{\alpha}$ we obtain
	\begin{align}
		\label{eq:upper:bound:testing:2}
		\trisk[  \rwdclass, A \msera]  \leq	\trisk[ \test_{\alpha, \optdim} \mid \rwdclass, A \msera] \leq \alpha,
	\end{align}
	i.e. $\msera^2$	is an upper bound for the minimax radius of testing. 
\end{corollary}

	\begin{illustration}
  		\label{ill:test}
  		We illustrate the order of the upper bound for the radius of testing $\msera^2$ derived in \cref{ub:msera} under the typical
  		smoothness and ill-posedness assumptions introduced in \cref{ill:qf}. Comparing the next table with \cref{ill:qf}, we emphasize that there is no elbow effect.  The derivation of the bounds is similar to the ones in \cref{ill:qf} and is thus omitted.\\ [3ex]
  		\centerline{\begin{tabular}{ll|l}\toprule
  				\multicolumn{3}{c}{Order of the minimax radius of testing $\msera^2$}\\\midrule
  				${\wdclass[j]}$ & ${\lv \epsden[j] \rv}$ & $\msera^2$ \\
  				(smoothness) & (ill-posedness) & \\ \midrule
  				${ j^{-\sPara}}$ & ${\lv j \rv^{-\pPara}}$ & $  \say^{-\tfrac{4\sPara}{4\sPara + 4 \pPara + 1}}$ \\
  				${ j^{-\sPara}}$ & 	$ {e^{-\lv j \rv^{\pPara}}}$ & $(\log\say)^{-\tfrac{2\sPara}{\pPara}}$ \\
  				${e^{-j^{\sPara}}}$ & ${\lv j \rv^{-\pPara}}$ & $\say^{-1} ()\log\say)^{\tfrac{4 \pPara + 1}{2\sPara}}$ \\
  				\bottomrule
  		\end{tabular}}
  	\end{illustration}
\section{Lower bound for the radius of testing}
\label{sec:lower:test}
	In this section we prove a matching lower bound for the radius of testing. The proof is inspired by Assouad's cube technique (see \cite{Tsybakov2009}, Chapter 2.7 for an explanation of the technique in the estimation case), where the testing risk is reduced to a distance between probability measures. It requires the construction of $2^{\optdim}$ candidates (called hypotheses) in the class $\rwdclass$, which are vertices on a hypercube. Roughly speaking, they are constructed such that they are statistically indistinguishable from the null $\xoden$, while having largest possible $\Lp[2]$-distance. 
  		\begin{proposition}[Lower bound for the radius of testing]
  			\label{testing:lb}
  			Assume $\ssconst := 2 \sum_{j \in \IN} \wdclass[ j ]^2< \infty$ and let $\eta \in (0,1]$ satisfy 
   			\begin{align}
   				\label{eta}
 \lb \wdclass[\optdim]^2 \vee \nu_{\optdim}^2 \rb \eta \leq \lb \wdclass[\optdim]^2 \wedge \nu_{\optdim}^2 \rb.
  			\end{align}
  			For $\alpha \in (0,1)$ define $\uA_\alpha^2 := \eta \lb \rdclass^2 \wedge \sqrt{ \log(1+2\alpha^2)} \wedge {\ssconst}^{-1} \rb$. Then, for all $A  \leq \uA_\alpha$
  			\begin{align*}
  			 \trisk[\rwdclass, A {\msera[n]}] \geq 1 - \alpha,
  			\end{align*}
  		i.e. $\msera[n]^2$ is a lower bound for the minimax radius of testing.
  		\end{proposition}
  		\begin{proof}[Proof of \cref{testing:lb}] 
  			 \textbf{Reduction Step.} To prove a lower bound for the testing radius we reduce the risk of a test to a distance between probability measures. Denote $\IP_{0} := \Pf[\xoden]$ and let $\IP_{1}$, specified below, be a mixing measure over the $\uA_\alpha \msera[n]$-separated alternative. The minimax risk can then be lower bounded by applying a classical reduction argument as follows
  				\begin{align*}
  					\trisk[\rwdclass, \uA_\alpha {\msera[n]}] & \geq \inf_{\test} \lb \IP_{0}(\test = 1) +  \IP_{1}(\test = 0) \rb = 1 - \TV(\IP_0, \IP_1)  \geq 1 - \sqrt{\frac{\chisq(\IP_0,\IP_1)}{2}}, 
  				\end{align*}
  				where $\TV$ denotes the total variation distance and $\chisq$ the $\chi^2$-divergence. The last inequality follows e.g. from Lemma 2.5 combined with (2.7) in \cite{Tsybakov2009}. \\
  				 \textbf{Definition of the mixtures.} On the alternative, we mix the Fourier coefficients uniformly over the vertices of a hypercube.  Consider $\xden \in \rwdclass \cap \Lp[2]_{\uA_\alpha \msera}$ with coefficients $\xden[0] = 1$, $\xden[\expan] = 0$ for $\lv \expan \rv > \optdim$ and
  				 \begin{align*}
  				 	\xden[ \expan ] :=\frac{ \sqrt{\zeta \eta} \msera[n] }{\sqrt{\sum_{0 < \lv l \rv \leq \optdim} \lv \epsden[l]\rv^{-4}} }  {\lv \epsden[\expan]\rv^{-2}}  \qquad \text{for } 0 < \lv \expan \rv \leq \optdim
  				 \end{align*}
  				 with $\zeta = \rdclass^2 \wedge \sqrt{  \log(1+ 2 \alpha^2)} \wedge {\ssconst}^{-1}$. For a sign vector $\signv \in \lcb \pm  \rcb^{\optdim}$, we define $\xden^{\signv} \in \rwdclass \cap \Lp[2]_{\uA_\alpha \sera}$ through its Fourier coefficients $\xden[0]^{\signv} = 1$, $\xden[j]^{\signv} = \signv[\lv \expan \rv] \xden[\expan]$ for $0 < \lv \expan \rv \leq \optdim$ and $\xden[j]^{\signv} = 0$ otherwise. The quadratic functionals $\qF(\xden^{\signv}) = \qF(\xden)$ and $\qF_k(\xden^{\signv}) = \qF_k(\xden)$, $k \in \IN$ are invariant under $\signv$. The resulting mixing measure is given by $
  					\IP_1 :={2^{-\optdim}} \sum_{\signv \in \lcb \pm  \rcb^{\optdim}} \Pf[\xden^{\signv}]$. Summarizing, $\xden^\tau$ satisfies:
  					\begin{enumerate}
  						\item[(a)] $ \sum_{\expan \in \IZ} \lv \xden[\expan]^\tau \rv^2 < \infty$, for all $\tau \in \{ \pm \}^{\optdim}$, by construction. \hfill ($\in \Lp[2]$)
  						\item[(b)] $\xden[\expan]^\tau = \overline{\xden[-\expan]^\tau}$, for all $\tau \in \{\pm \}^{\optdim} $, by construction.  \hfill (real-valued)
  						\item[(c)] $\xden[0]^\tau = 1$, for all $\tau \in \{\pm \}^{\optdim} $, by construction. \hfill (normalized to $1$)
  						\item[(d)] $\sum_{\lv \expan \rv > 0} \lv \xden[\expan]^\tau \rv   \leq 1$, for all $\tau \in \{\pm \}^{\optdim}$, since   \hfill (positive) \\
  						by the Cauchy-Schwarz inequality, since $\sum\limits_{\lv \expan \rv > 0} \lv \xden[\expan]^\tau \rv \leq \sqrt{\sum\limits_{\lv j \rv > 0} \wdclass[\lv j \rv]^2} \sqrt{\sum\limits_{\lv j \rv > 0} \wdclass[\lv j \rv]^{-2} \lv \xden[j] \rv^2} \leq \sqrt{\zeta} \sqrt{\ssconst} \leq 1$, where the second last inequality follows as in (e).
  						\item[(e)] $\xden \in \rwdclass$, i.e. $2 \sum_{\expan \in \IN} \wdclass[\expan]^{-2} \lv \xden[\expan] \rv^2 \leq \rdclass^2$, by the  monotonicity of $\wdclass[]$, since \hfill (smoothness) \\			
  						$2 \sum_{\expan \in \IN}  \wdclass[\lv j \rv ]^{-2} \lv \xden[\expan] \rv^2 				 \leq   \frac{\zeta \eta \msera[n]^2 \wdclass[\optdim]^{-2}}{ {\sum_{0 < \lv l \rv \leq \optdim} \lv \epsden[l]\rv^{-4}}} \sum_{0 < \lv \expan \rv \leq \optdim} \lv \epsden[\expan] \rv^{-4}  = \zeta \eta \msera[n]^2 \wdclass[\optdim]^{-2} \leq \zeta \leq \rdclass^2$.
  						\item[(f)]  $ f  \in  \Lp[2]_{\uA_\alpha \msera[n]}$, i.e. $\qFr_{\optdim}(\xden ) \geq  \uA_\alpha \msera[n]$, since  \hfill (separation) \\
  						$\qF_{\optdim}(\xden ) = \frac{\zeta \eta \msera[n]^2}{ {\sum_{0 < \lv l \rv \leq \optdim} \lv \epsden[l]\rv^{-4}}} \sum_{0 \leq \lv \expan \rv < \optdim} \lv \epsden[\expan] \rv^{-4} = \zeta \eta \msera[n]^2 = \uA_\alpha^2  \msera[n]^2. $
  						\item[(g)] $n^2 \sum\limits_{0 < \lv \expan \rv \leq \optdim} \lv \xden[\expan] \rv^4 \lv \epsden[\expan] \rv^4  \leq \log(1+2\alpha^2)$, since \hfill (similarity) \\
  						$n^2 \sum\limits_{0 < \lv \expan \rv \leq \optdim} \lv \xden[\expan] \rv^4 \lv \epsden[\expan] \rv^4 =  \zeta^2 \eta^2 \msera[n]^4 \frac{\sum_{0 < \lv l \rv \leq \optdim} \lv \epsden[l]\rv^{-4}}{\lb \sum_{0 < \lv l \rv \leq \optdim} \lv \epsden[l]\rv^{-4} \rb^2} \leq \zeta^2 \leq  \log(1+ 2 \alpha^2)$. 
  					\end{enumerate}
  					The conditions (a)-(d) guarantee that the vertices are densities, (e) and (f) guarantee that the vertices lie in the alternative. \\
  \textbf{Bound of the $\chisq$-divergence.}
  				We apply \cref{lemma: mixing over cubes densities} in the appendix and obtain
  				\begin{align*}
  					\chisq\lb  \frac{1}{2^{\optdim}} \sum_{\signv \in \lcb \pm  \rcb } \Pf[\xden^{\signv}] , \IP_0 \rb \leq \exp\lb 2 n^2 \sum_{\expan =1}^{\optdim} \lv \yden[\expan] \rv^4 \rb - 1 =  \exp\lb  n^2 \sum_{0 < \lv \expan \rv \leq \optdim} \lv \xden[\expan] \rv^4 \lv \epsden[\expan] \rv^4 \rb - 1 
  				\end{align*}
  				Hence, (g) guarantees that the induced distance between the mixing measure and the null is negligible. Combined with the reduction step, it follows $ 	\trisk[\rwdclass, \uA_\alpha {\msera[n]}]  \geq 1 - \alpha$. 
  		\end{proof}
  	
  	  	\begin{remark}[Conditions on $\eta$ and $\ssconst$.]
  		\cref{testing:lb} involves the value $\eta$ satisfying \eqref{eta}, which depends on
  		the joint behaviour of the sequences $\lcb \wdclass[j] \rcb_{j \in \IN}$ and $\lcb \epsden[j] \rcb_{j \in \IZ}$ and essentially guarantees an optimal balance of
  		the bias and the variance term in the dimension $\optdim$. For all the typical smoothness and ill-posedness assumptions considered in \cref{ill:test} an $\eta$ exists such that \eqref{eta} holds uniformly for all $n \in \IN$.  	The additional assumption $\ssconst = 2 \sum_{j \in \IN} \wdclass[j]^2 < \infty$ in \cref{testing:lb} is needed to ensure that the candidate densities constructed in the reduction scheme of the proof are indeed densities. This assumption is in particular satisfied for the typical smoothness classes introduced in \cref{ill:qf}. For Sobolev-type alternatives, i.e. $\wdclass[j] \sim j^{-2 \sPara}$, $j \in \IN$  it is satisfied as soon as $s > 1/2$, for analytic alternatives, i.e. $\wdclass[j] \sim \exp(-j^\sPara)$, $j \in \IN$ it is satisfied for all positive $\sPara$. 
  	\end{remark}

 \section{Lower bound for the estimation risk}
		In this section we first explore the connection between quadratic functional estimation and testing. Every estimator for the functional $\qF(f) = \lV \xoden - \xden \rV_{\Lp[2]}^2$ can be used to construct a test by rejecting the null as soon as the estimated value of the quadratic functional exceeds a certain threshold. The next proposition shows how this connection can be formalized in terms of the minimax risk and the minimax radius.
   			\label{sec:testing:vs:estimation}
  		\begin{proposition}[Testing is faster than quadratic functional estimation] 
  			\label{testing:vs:estimation} 
  			Let $\alpha \in (0,1)$, $\mc E \subset \Lp[2]$ be a class of functions and $\sera^2(\mc E)$ a minimax radius of testing with $\uA_\alpha$ as in the lower bound definition. Then, the minimax risk of estimation satisfies
  			\begin{align*}
  				r^2(\mc E) \geq (1-\alpha) \frac{\uA_\alpha^2}{8} \cdot \sera^4(\mc E).
  			\end{align*}
  		\end{proposition}
		\begin{proof}[Proof of 	\cref{testing:vs:estimation}]
				Let $\hqF$ be any estimator of $\qF(\xden)$. Define the test $\test := \Ind{\lcb \hqF \geq \rho/2 \rcb }$ with $\rho =  \uA_\alpha \sera(\mc E)$. We convert the mean squared error into the sum of type I and type II error probabilities, i.e. the testing risk, by applying Markov's inequality. Keeping in mind that $\qF(\xoden) = 0$, we have
			\begin{align*}
				\erisk[\hqF, \mc E] & =  \sup_{\xden \in \mc E} \Ef[\xden] \lb \hqF - \qF(\xoden) \rb^2  \geq \frac{1}{2} \lcb \Ef[\xoden] \lb \hqF \rb^2  + \sup_{\xden \in \mc E \cap \Lp[2]_\rho}\Ef[\xden] \lb \hqF - \qF(\xden) \rb^2\rcb \nonumber \\
				& \geq \frac{\rho^4}{8} \lcb \Pf[\xoden] \lb \hqF  \geq \tfrac{\rho^2}{2}\rb  + \sup_{\xden \in \mc E \cap \Lp[2]_\rho}\Pf[\xden] \lb  \qF(\xden) -  \hqF  \geq \tfrac{\rho^2}{2} \rb\rcb \nonumber \\
				& \geq \frac{\rho^4}{8} \lcb \Pf[\xoden] \lb \hqF \geq \tfrac{\rho^2}{2}\rb  + \sup_{\xden \in \mc E \cap \Lp[2]_\rho}\Pf[\xden] \lb \hqF   \leq \tfrac{\rho^2}{2} \rb\rcb  = \frac{\rho^4}{8} \trisk[\test \mid \mc E,  \uA_\alpha \sera(\mc E)]. 
			\end{align*}
		Since $\hqF$ is arbitrary and by definition
			$\trisk[\mc E, \uA_\alpha \sera(\mc E)] \geq 1 - \alpha$, we obtain the result.
		\end{proof}  
 Recall that the upper bound for the risk of estimation in \eqref{re:ub:est} is of order $\msera^4 \vee \baseterm$.
	There are two possible scenarios, either the risk is governed by the term $\msera^4 = \min_{k \in \IN}	\lcb  \wdclass[k]^4 \vee \nu_k^4 \rcb$ or by the baseterm  $\baseterm = \max_{m \in \IN}  \lcb \wdclass[m]^4 \lb 1 \wedge \frac{1}{n \wdclass[m]^{2} \lv \epsden[m] \rv^{2} } \rb \rcb$. We prove separate lower bounds for these two cases. The lower bound in the first case is an immediate consequence of \cref{testing:vs:estimation} combined with \cref{testing:lb} and we omit its proof. 
			\begin{corollary}[First lower bound for the risk of estimation]
				Let $\eta \in (0,1]$ satisfy \eqref{eta}. Then, for all $n \geq 2$
				 \begin{align*}
				 	\erisk[\rwdclass] \geq  \frac{ \eta^2 \lb \rdclass^4 \wedge {\log(3/2)}\rb }{16}  \min_{k \in \IN}	\lcb  \wdclass[k]^4 \vee \frac{1}{\say^2} \sum_{ 0 < \lv \expan \rv \leq k} \frac{1}{\lv \epsden[\expan]\rv^{4}}  \rcb.
				 \end{align*}
			\end{corollary}
			In contrast to the lower bound proved in \cref{testing:lb}, the proof of the next proposition only requires the construction of two candidate densities. 
			
				\begin{proposition}[Second lower bound for the risk of estimation]
				\label{lower:bound:est:inter}
				For all $n \geq 2$ we have
				\begin{align*}
					\erisk[\rwdclass] \geq  \lb\frac{1}{64} \wedge \frac{\rdclass^4}{16} \rb  \max_{m \in \IN}  \lcb \wdclass[m]^4 \lb 1 \wedge \frac{1}{n \wdclass[m]^{2} \lv \epsden[m] \rv^{2} } \rb \rcb.
				\end{align*}
			\end{proposition}
								\begin{proof}[Proof of \cref{lower:bound:est:inter}]
					\textbf{Reduction Step.} 	Denoting by $\IQ_{\xden}$ the measure with density $f \ccon \varphi$, the measure $\Pf[{\xden}]$  associated with the observations equals the n-fold product measure of $\IQ_{\xden}$. Let $\xden^+, \xden^- \in \densities$ (to be specified below) with associated $\Pf[{\xden^+}]$, $\Pf[{\xden^-}]$ and quadratic functionals $\pfun =\qF( \xden^+)$ and $\qfun = \qF( \xden^-)$. Denote by $	\hela(\Pf[{\xden^+}], \Pf[{\xden^-}])  = \int \sqrt{\dif \Pf[{\xden^+}] \dif \Pf[{\xden^-}]} $ the Hellinger affinity between the two measures $\Pf[{\xden^+}]$ and $\Pf[{\xden^-}]$. We apply the reduction scheme in \cref{lemma:reduction:estimation} and obtain
					\begin{align}
						\label{reduc1}
						\erisk[\rwdclass] & \geq \frac{1}{8} \hela^2(\Pf[\xden^+], \Pf  [\xden^-]) (\pfun- \qfun)^2. 
					\end{align}
					Using the tensorization property of the Hellinger affinity and the definition of the  Hellinger distance (cp. for instance \cite{Tsybakov2009}, p. 83), it follows
				$	 \hela(\Pf[\xden^+], \Pf  [\xden^-]) = 	\lb \hela( \IQ_{\xden^+}, \IQ_{\xden^-} ) \rb^n = \lb 1 - \frac{1}{2} \hell^2(\IQ_{\xden^+},\IQ_{\xden^-})\rb^n$.
					Denoting  $\yden^{\pm} := \xden^{\pm} \ccon \epsden$, we will ensure that $\yden^- \geq \tfrac{1}{2}$ and $\lV \yden^+ - \yden^- \rV_{\Lp[2]} \leq 1$. Hence,
						\begin{align*}
						\hell^2( \IQ_{\xden^+},  \IQ_{\xden^+})  = \int \frac{\lb \yden^+(x) - \yden^-(x) \rb^2}{ \lb \sqrt{\yden^+(x)} + \sqrt{\yden^-(x)} \rb^2} \dif x  \leq 2 \lV \yden^+  - \yden^- \rV_{\Lp[2]}^2
					\end{align*}
					and by Bernoulli's inequality $ \hela^2(\Pf[\xden^+], \Pf  [\xden^-]) \geq 1 - 2n \lV \yden^+ - \yden^- \rV_{\Lp[2]}^2$. From \eqref{reduc1} it follows
					\begin{align}
						\erisk[\rwdclass] & \geq \frac{1}{8} (\pfun- \qfun)^2 \lb 1 - 2 n \lV \xden^+ \ccon \epsden - \xden^- \ccon \epsden \rV_{\Lp[2]}^2 \rb. \label{reduc}
					\end{align}
					\textbf{Construction of the hypotheses $\xden^+, \xden^-$.}  Let $\tau \in \lcb \pm  \rcb$ and let $m$ be arbitrary. Define the Fourier coefficients of the hypotheses $\xden^\tau$, $\tau \in \lcb \pm  \rcb$ by
					\begin{align*}
						\xden[\expan]^+ = \begin{cases}
							1 & \expan = 0 \\
							(1 + \x) C \y & \expan = \pm m \\
							0 & \text{otherwise}
						\end{cases} \qquad \text{and} \qquad 	\xden[\expan]^- = \begin{cases}
						1 & \expan = 0 \\
						(1 - \x) C \y & \expan = \pm m \\
						0 & \text{otherwise}
					\end{cases}
					\end{align*}
					with $C := \tfrac{1}{4} \wedge \tfrac{\rdclass}{\sqrt{8}}$ and $\x^2 := 1 \wedge \frac{1}{\say \wdclass[m]^2 \lv \epsden[m]\rv^2}$. Then, the hypotheses $\xden^\tau$,  $\tau \in \{\pm \}$ satisfy the following conditions:
					\begin{enumerate}
						\item \textbf{$\xden^\tau \in \densities$},
						\begin{enumerate}
							\item$  \sum_{\expan \in \IZ} \lv \xden[\expan]^\tau \rv^2 < \infty$, by construction, \hfill ($\in \Lp[2]$)
							\item $\xden[\expan]^\tau = \overline{\xden[-\expan]^\tau}$, by construction,  \hfill (real-valued)
							\item $\xden[0]^\tau = 1$, by construction. \hfill (normalized to $1$)
							\item $\sum_{\lv \expan \rv > 0} \lv \xden[\expan]^\tau \rv  = 2 (1 \pm \x) C \y \leq 2 \cdot 2 C \y \leq 4 C  \leq 1$ \hfill (positive)
							\item $\sum_{\lv \expan \rv > 0} \lv \xden[\expan]^- \rv \lv \epsden[\expan] \rv  = 2 (1- \x) C \y \lv \epsden[m] \rv \leq 2 C \leq \frac{1}{2}$  \hfill(bounded from below)
						\end{enumerate}
						\item \textbf{$\xden^\tau \in \rwdclass$} 
						\begin{enumerate}
							\item[(f)] 		$2 \sum_{\expan \in \IN }  \wdclass[ j]^{-2} \lv \xden[\expan]^\tau \rv^2 = 2 \wdclass[m]^{-2} (1 \pm \x)^2  C^2 \y^2 \leq 8 C^2 \leq \rdclass^2$ \hfill (smoothness)
						\end{enumerate}
						\item We have $\qF(\xden^{\tau}) = \sum\limits_{\lv \expan \rv > 0} \lb \xden[\expan]^{\tau} \rb^2 = 2 (1 \pm \x)^2 \y^2$, therefore 
						\begin{enumerate}
							\item[(g)] $	(\pfun- \qfun)^2 =4 \lb (1 + \x)^2  - (1 - \x)^2 \rb^2 C^4 \y^4 = 64 \x^2 C^4 \y^4 $. \hfill (separation) 
						\end{enumerate}
						\item
						\begin{enumerate}
							\item[(h)] 	$\lV \xden^+ \ccon \epsden - \xden^- \ccon \epsden \rV_{\Lp[2]}^2  = 4 C^2 \x^2 \y^2 \lv \epsden[m] \rv^2 
							\leq 4 C^2 \frac{1}{n} \leq \frac{1}{4n}$ \hfill (similarity)
						\end{enumerate}
					\end{enumerate}
					Note that condition (e) guarantees that $\xden^{-} \ccon \epsden \geq \tfrac{1}{2}$, Condition (h) implies  $\VnormLp[2]{\xden^+ \ccon \epsden - \xden^- \ccon \epsden }^2 \leq 1$, which is a condition to apply Bernoulli's inequality. Combining the bounds $(g)$ and $(h)$ with the reduction \eqref{reduc}, we obtain $	\erisk[\rwdclass]  = 4 C^4 \wdclass[m]\lb 1 \wedge \frac{1}{\say \wdclass[m]^2 \lv \epsden[m]\rv^2}\rb$ for all $m \in \IN$, which proves the assertion.
				\end{proof}

\appendix
\section{Appendix}
\label{appendix}
\subsection{Auxiliary results for proving lower bounds of testing}

\begin{lemma}[\textbf{Switching sums and products on cubes}] 
	\label{induction}
	For $k \in \IN$ let $J_j^+$, $J_j^-$, $j \in \lcb 1, \dots, k \rcb$ be real numbers. Then,
	\begin{align*}
		\frac{1}{2^{k}}\sum_{\tau \in \{\pm \}^{k}} \prod_{j=1}^{k} J_j^{\tau_j}
		& = \prod_{j=1}^{k} \frac{J_j^{-}+J_j^{+}}{2}.
	\end{align*}
\end{lemma}
\begin{proof}[Proof of \cref{induction}]
	The proof is by induction on $k$. The \textbf{base case} $k=1$ follows immediately.
	For the \textbf{induction step}, assume $\frac{1}{2^{k}}\sum_{\tau \in \{\pm \}^{k}} \prod_{j=1}^{k} J_j^{\tau_j}= \prod_{j=1}^{k} \frac{J_j^{-}+J_j^{+}}{2}$. Then,
	\begin{align*}
	 &	\frac{1}{2^{k+1}}\sum_{\tau \in \{\pm \}^{k+1}} \prod_{j=1}^{k+1} J_j^{\tau_j} 
		= 	\frac{1}{2^{k+1}}\lb \lb \sum_{\substack{\tau \in \{\pm \}^{k} }} \prod_{j=1}^{k} J_j^{\tau_j} \rb \cdot J_{k+1}^{+} + \lb \sum_{\substack{\tau \in \{\pm \}^{k}}} \prod_{j=1}^{k} J_j^{\tau_j}\rb  \cdot J_{k+1}^{-} \rb \\ 
		& =	\frac{1}{2}\lb J_{k+1}^{+} +  J_{k+1}^{-} \rb  \lb \frac{1 }{2^k}\sum_{\substack{\tau \in \{\pm \}^{k} }} \prod_{j=1}^{k} J_j^{\tau_j}  \rb 
		= 	\frac{1}{2}\lb J_{k+1}^{+} +  J_{k+1}^{-} \rb \prod_{j=1}^{k} \frac{J_j^{-}+J_j^{+}}{2}  = \prod_{j=1}^{k+1} \frac{J_j^{-}+J_j^{+}}{2},
	\end{align*}
	where the induction assumption was used in the second last step.
\end{proof}

\begin{lemma}[\textbf{$\chi^2$-divergence for mixtures over hypercubes}]
	\label{lemma: mixing over cubes densities}
	Let $k \in \IN$. For $\tau \in \{\pm \}^k$ and $\theta \in \ell^2(\IN)$ we define the coefficients $\theta^\tau = (\theta_j^\tau)_{j \in \IZ}$ and functions  $f^\tau \in \Lp[2]$ by setting
		\begin{align*}
		\theta_j^{\tau} := \begin{cases}
			\tau_j \theta_j & \lv j \rv \in \{1, \dots, k \} \\
			1 & j = 0 \\
			0 & \lv j \rv > k
		\end{cases} \qquad \text{and} \qquad
		f^\tau := \sum_{j=-k}^k \theta_j^\tau \expb = \mathds{1}_{[0,1]} + \sum_{0< \lv j \rv \leq k} \theta_{j}^\tau \expb. 
	\end{align*}
	Assuming $f^\tau \in \densities$ for each $\tau \in \{\pm\}$, we consider the mixing measure $\IP_\mu$ with probability density $ \frac{1}{2^\kappa} \sum_{\tau \in \{\pm \}^\kappa}\prod_{i=1}^n f^\tau(z_i)$, $z_i \in [0,1)$, $i \in \{1, \dots, n\}$, and denote $\IP_0 = \Pf[\xoden]$.
	Then, the $\chi^2$-divergence satisfies
	\begin{align*}
		\chi^2(\IP_\mu, \IP_0) \leq { \exp\lb 2 n^2 \sum_{j=1}^\kappa \theta_j^4 \rb} - 1.
	\end{align*}
\end{lemma}

\begin{proof}[Proof of \cref{lemma: mixing over cubes densities}] Recall that $	\chi^2(\IP_\mu, \IP_0) = \IE_0 \lb	\frac{\dif \IP_\mu}{\dif \IP_{0}} (Z_1, \dots, Z_n) \rb^2 - 1$ for random variables $(Z_j)_{j \in \IN}$ with marginal density $f_\circ = \mathds{1}_{[0,1]}$ under $\IP_0$.
Let $z_1,\dots, z_n \in [0,1)$, then the likelihood ratio  becomes
\begin{align*}
	\frac{\dif \IP_\mu}{\dif \IP_{0}} (z_1, \dots, z_n)= \frac{1}{2^k} \sum_{\tau \in \{\pm \}^k}\prod_{i=1}^n f^\tau(z_i).
\end{align*} 
Squaring, taking the expectation under $\IP_0$ and exploiting the independence yields
\begin{align*}
	\IE_0 \lb	\frac{\dif \IP_\mu}{\dif \IP_{0}} (Z_1, \dots, Z_n) \rb^2 
	= \lb \frac{1}{2^k} \rb^2 \sum_{\eta, \tau \in \{\pm \}^k} \lb \IE_0 \lb f^\tau(Z_1) f^\eta(Z_1) \rb \rb^n.
\end{align*}
Let us calculate
$	\IE_0 \lb  f^\tau(Z_1) f^\eta(Z_1) \rb = \int_{[0,1]} f^\tau(z) f^\eta(z) \dif z =   1 +  2 \sum_{ j = 1}^{k} \theta_j^\tau \theta_j^\eta$,
where the last equality is due to the orthonormality of $(e_j)_{j \in \IZ}$ and the symmetry of $\theta^\tau$ and $\theta^\eta$. 
Applying the inequality $1+x \leq \exp(x)$, which holds for all $x \in \IR$, we obtain
\begin{align*}
	\IE_0 \lb f^\tau(Z_1) f^\eta(Z_1) \rb \leq \lb 1 + 2  \sum_{ j=1}^k \theta_j^\tau \theta_j^\eta \rb \leq  \exp \lb  2 \sum_{j=1}^{ k} \theta_j^\tau \theta_j^\eta  \rb = \prod_{j=1}^k \exp \lb 2  \theta^\tau_j \theta^\eta_j \rb.
\end{align*}
Hence, 
\begin{align*}
	\IE_0 \lb	\frac{\dif \IP_\mu}{\dif \IP_{0}} (Z_1, \dots, Z_n) \rb^2 & \leq
	 \lb \frac{1}{2^k} \rb^2 \sum_{\eta, \tau \in \{\pm \}^k} \prod_{ j=1}^k \exp \lb 2 n  \theta^\tau_j \theta^\eta_j \rb, 
\intertext{where we can apply \cref{induction} to the $\eta$-summation with $J_j^\eta = \exp \lb  2 n  \theta^\tau_j \theta^\eta_j \rb$ and obtain }
 	\IE_0 \lb	\frac{\dif \IP_\mu}{\dif \IP_{0}} (Z_1, \dots, Z_n) \rb^2  & \leq	\lb \frac{1}{2^k} \rb \sum_{\tau \in\{\pm  \}^k} \lb  \prod_{j=1}^k \frac{\exp \lb - 2 n  \theta_j \theta^{\tau}_j \rb +\exp \lb  2 n  \theta_j \theta^{\tau}_j \rb }{2}\rb. 
\intertext{Again applying \cref{induction} now to the $\tau$-summation with $J_j^\tau = \tfrac{\exp \lb - 2 n  \theta_j \theta^{\tau}_j \rb +\exp \lb  2 n  \theta_j \theta^{\tau}_j \rb }{2} $ yields}
	\IE_0 \lb	\frac{\dif \IP_\mu}{\dif \IP_{0}} (Z_1, \dots, Z_n) \rb^2 	& \leq \prod_{j=1}^k  \  \frac{\exp \lb - 2 n  \theta_j^2 \rb +\exp \lb 2 n  \theta_j^2 \rb + \exp \lb 2 n  \theta_j^2  \rb +\exp \lb -  2 n  \theta_j^2  \rb}{4}  \\
	& =  \prod_{j=1}^k\lb  \frac{\exp \lb - 2 n  \theta_j^2 \rb +\exp \lb + 2 n  \theta_j^2 \rb }{2} \rb 
	 = \prod_{j=1}^k \cosh\lb 2 n  \theta_j^2   \rb.
	\intertext{Since $\cosh(x) \leq \exp(x^2/2)$, we obtain}
		\IE_0 \lb	\frac{\dif \IP_\mu}{\dif \IP_{0}} (Z_1, \dots, Z_n) \rb^2 & \leq  \prod_{j=1}^k\exp\lb  2 n^2  \theta_j^4  \rb =  \ \exp\lb  2 n^2 \sum_{j=1}^k  \theta_j^4 \rb,
\end{align*}
which completes the proof.
\end{proof}

\subsection{Auxiliary results for proving lower bounds of estimation}
\begin{lemma}[\textbf{Reduction scheme for the estimation risk}]
	\label{lemma:reduction:estimation}
	For densities $f^+, f^- \in \mc E \subseteq \densities$ we have
	\begin{align}
		\label{reduction:testing}
		\inf_{\hqF} \sup_{\xden \in \mc E} \Ef[\xden] \lb \hqF - \qF(f) \rb^2 \geq \frac{1}{8} \hela^2(\Pf[f^+], \Pf[f^-]) ( \qF(f^+) -\qF(f^-))^2,
	\end{align}
	where $\hela(\Pf[f^+], \Pf[f^-])$ denotes the Hellinger affinity between $ \Pf[f^+]$ and $\Pf[f^-]$.
\end{lemma}
\begin{proof}[Proof of \cref{lemma:reduction:estimation}]
	Let $\hqF$ be any estimator and denote $\Pf[+] := \Pf[f^+]$, $\Pf[-] = \Pf[f^-]$ and $\qfun =\qF(f^+)$, $\pfun = \qF(f^-)$.  We have
	\begin{align*}
	\hela(\IP_+, \IP_-) & = \int \sqrt{\dif \IP_+ \dif \IP_-} = \int \lv \frac{\qfun - \pfun}{\qfun - \pfun} \rv \sqrt{\dif \IP_+ \dif \IP_-} \\
	& \leq \lb \int \lv \frac{\qfun - \hqF}{\qfun - \pfun} \rv^2 {\dif \IP_+} \rb^{\tfrac{1}{2}} \lb \int \dif \IP_- \rb^{\tfrac{1}{2}}  +\lb  \int \lv \frac{\hqF - \pfun}{\qfun - \pfun} \rv { \dif \IP_-} \rb^{\tfrac{1}{2}} \lb \int \dif \IP_+ \rb^{\tfrac{1}{2}} \\
	& \leq 2 \lv \qfun - \pfun \rv^{-1} \lb \IE_{f^+} \lb \hqF - \qfun \rb^2 + \IE_{f^-} \lb \hqF - \pfun \rb^2 \rb^{\tfrac{1}{2}}.
	\end{align*}
	Therefore,
	\begin{align*}
		 \sup_{\xden \in \mc E} \Ef[\xden] \lb \hqF - \qF(f) \rb^2  \geq \frac{1}{2} \lb \IE_{f^+} \lb \hqF - \qfun \rb^2 + \IE_{f^-} \lb \hqF - \pfun \rb^2 \rb \geq \frac{ \hela^2(\IP_+,\IP_-) }{8}\lb \qfun - \pfun \rb^2,
	\end{align*}
	which completes the proof.
\end{proof}

\subsection{Calculations for the risk bounds in \cref{ill:qf}}
\label{calculations}
	We determine the order of the terms $\baseterm$ and $\msera^4$ in \eqref{opt:ub} for each of the three combinations in  \cref{ill:qf} and determine the dominating term. Let $m_\star =  \max  \lcb m \in \IN: \wdclass[m]^4 \geq  \frac{ \wdclass[m]^2 }{\say  \lv \epsden[m] \rv^{2} }  \rcb$.
\begin{enumerate}
\item \textbf{(ordinary smooth - mildly ill-posed)} Consider first $\msera^4$ defined in \eqref{baselevel}. The variance term $\nu_k^4 = \frac{1}{\say^2} \sum\limits_{ 0 < \lv \expan \rv \leq k} \frac{1}{\lv \epsden[\expan]\rv^{4}} \sim \frac{1}{\say^2} \sum\limits_{ 0 < \lv \expan \rv \leq k} \lv \expan \rv^{4 \pPara}$ is of order $\frac{1}{\say^2} k^{4 \pPara + 1}$ and the bias term $\wdclass[k]^4$ is of order $k^{-4 \sPara}$. Hence, the optimal $\optdimest$ satisfies $\optdimest^{-4 \sPara} \sim \frac{1}{n^2} \optdimest^{4 \pPara}$ and thus $\optdimest \sim n^{\frac{2}{4\pPara + 4\sPara + 1}}$, which yields an upper bound of order $\mrate^2 \sim \optdimest^{-4 \sPara} \sim  n^{-\frac{8s}{4\pPara + 4\sPara + 1}}$. For the base level $\baseterm =  \max_{m \in \IN}  \lcb  \wdclass[m]^4 \wedge  \frac{ \wdclass[m]^2 }{\say  \lv \epsden[m] \rv^{2} }  \rcb$, the term  $\frac{ \wdclass[m]^2 }{\say  \lv \epsden[m] \rv^{2} } \sim \frac{1}{n} m^{2(\pPara- \sPara)}$ is monotonically increasing in $m$ for $\pPara - \sPara > 0$ and monotonically non-increasing otherwise. Let $\pPara - \sPara > 0$, then $m_\star$ satisfies $m_\star^{-4 \sPara} \sim \frac{1}{n} m_\star^{-2(\sPara - \pPara)}$ and is thus of order $m_\star \sim n^{\frac{2\sPara}{\sPara+\pPara}}$. Therefore, $\baseterm \sim n^{-\frac{8\sPara}{4\sPara + 4\pPara}}$ is negligible compared with $\msera^4$. Let $\pPara - \sPara \leq 0$, then both $\wdclass[m]^4$ and $\frac{ \wdclass[m]^2 }{\say  \lv \epsden[m] \rv^{2} }$ are non-increasing. The maximum of their minimum is attained at $m_\star = 1$, which yields $\baseterm  \sim \frac{1}{n}$. Hence, $\baseterm$ is of larger order than $\msera^4$ for $s - p > \frac{1}{4}$ only.

\item \textbf{(ordinary smooth - severly ill-posed)}  Consider first $\msera^4$ defined in \eqref{baselevel}.  The variance term $\frac{1}{\say^2} \sum\limits_{ 0 < \lv \expan \rv \leq k} \frac{1}{\lv \epsden[\expan]\rv^{4}} \sim \frac{1}{\say^2} \sum\limits_{ 0 < \lv \expan \rv \leq k} \exp( \lv \expan \rv^{4 \pPara})$ is of order $\frac{1}{\say^2} \exp( 4 k^\pPara)$ and the bias term $\wdclass[k]^4$ is of order $k^{-4 \sPara}$. Hence, the optimal $\optdimest$ satisfies $\optdimest^{-4 \sPara} \sim \frac{1}{n^2} \exp( 4 \optdimest^\pPara)$ and thus $\optdimest \sim \log(n^2/b_n)^{\frac{1}{\pPara}}$ with $b_n \sim \log(n^2)^{\frac{4\sPara}{\pPara}}$, which yields an upper bound of order $\mrate^2 \sim \optdimest^{-4 \sPara} \sim  \log(n)^{-\frac{4\sPara}{\pPara}}$. Considering the base level $\baseterm = c_3 \max_{m \in \IN}  \lcb  \wdclass[m]^4 \wedge  \frac{ \wdclass[m]^2 }{\say  \lv \epsden[m] \rv^{2} }  \rcb$, the term $\frac{ \wdclass[m]^2 }{\say  \lv \epsden[m] \rv^{2} } \sim \frac{m^{-2 \sPara}}{n} \exp(2 m^\pPara)$ is monotonically increasing in $m$. Hence, $m_\star$ satisfies $m_\star^{-4 \sPara} \sim \frac{1}{n} m_\star^{-2\sPara} \exp(2 m ^\pPara)$ and thus $m_\star \sim \log(n/b_n)^{\frac{1}{\pPara}}$ with $b_n \sim \log(n)^{\frac{2 \sPara}{\pPara}} $. Therefore, $\baseterm \sim \log(n)^{-\frac{4 \sPara}{\pPara}}$ is of the same order as $\msera^4$. 

\item  \textbf{(super smooth - mildly ill-posed)} Consider first $\msera^4$ defined in \eqref{baselevel}. The term $\frac{1}{\say^2} \sum\limits_{ 0 < \lv \expan \rv \leq k} \frac{1}{\lv \epsden[\expan]\rv^{4}} \sim \frac{1}{\say^2} \sum\limits_{ 0 < \lv \expan \rv \leq k} \lv \expan \rv^{4 \pPara}$ is of order $\frac{1}{\say^2} k^{4 \pPara + 1}$, whereas the bias term $\wdclass[k]^4$ is of order $\exp(-4 k^\sPara)$. Hence, the optimal $\optdimest$ satisfies $\exp(-4 \optdimest^\sPara) \sim \frac{1}{n^2} \optdimest^{4 \pPara}$ and thus $\optdimest \sim \log(n^2/b_n)^{\frac{1}{\sPara}}$ with $b_n \sim \log(n)^{\frac{4 \pPara + 1}{\sPara}}$, which yields an upper bound of order $\mrate^2 \sim \frac{1}{n^2} \optdimest^{4 \pPara + 1} \sim  \frac{1}{n^2} \log(n)^{\frac{4\pPara+1}{\sPara}}$. Considering the base level $\baseterm = c_3 \max_{m \in \IN}  \lcb  \wdclass[m]^4 \wedge  \frac{ \wdclass[m]^2 }{\say  \lv \epsden[m] \rv^{2} }  \rcb$, the term $\frac{ \wdclass[m]^2 }{\say  \lv \epsden[m] \rv^{2} } \sim \frac{m^{2 \pPara}}{n} \exp(-2 m^\sPara)$ is monotonically decreasing in $m$. Hence, $m_\star \sim 1$. Therefore, $\baseterm \sim \frac{1}{n}$ is of larger order than $\msera^4$ and is thus the dominant term.
\end{enumerate}

\bibliography{lit.bib}

\begin{thebibliography}{28}
\providecommand{\natexlab}[1]{#1}
\providecommand{\url}[1]{\texttt{#1}}
\expandafter\ifx\csname urlstyle\endcsname\relax
  \providecommand{\doi}[1]{doi: #1}\else
  \providecommand{\doi}{doi: \begingroup \urlstyle{rm}\Url}\fi

\bibitem[Baraud(2002)]{Baraud2002}
Y.~Baraud.
\newblock Non-asymptotic minimax rates of testing in signal detection.
\newblock \emph{Bernoulli}, 8\penalty0 (5):\penalty0 577--606, 2002.

\bibitem[Bickel and Ritov(1988)]{BickelRitov1988}
P.~J. Bickel and Y.~Ritov.
\newblock Estimating integrated squared density derivatives: sharp best order
  of convergence estimates.
\newblock \emph{Sankhy{\=a}: The Indian Journal of Statistics, Series A}, pages
  381--393, 1988.

\bibitem[Birge and Massart(1995)]{BirgeMassart1995}
L.~Birge and P.~Massart.
\newblock Estimation of integral functionals of a density.
\newblock \emph{The Annals of Statistics}, 23\penalty0 (1):\penalty0 11--29, 02
  1995.

\bibitem[Butucea(2007)]{Butucea2007}
C.~Butucea.
\newblock Goodness-of-fit testing and quadratic functional estimation from
  indirect observations.
\newblock \emph{The Annals of Statistics}, 35\penalty0 (5):\penalty0
  1907--1930, 2007.

\bibitem[Butucea and Meziani(2011)]{ButuceaMeziani2011}
C.~Butucea and K.~Meziani.
\newblock Quadratic functional estimation in inverse problems.
\newblock \emph{Statistical Methodology}, 8\penalty0 (1):\penalty0 31--41,
  2011.

\bibitem[Cai and Low(2005)]{CaiLow2005}
T.~T. Cai and M.~G. Low.
\newblock Nonquadratic estimators of a quadratic functional.
\newblock \emph{The Annals of Statistics}, 33\penalty0 (6):\penalty0
  2930--2956, 2005.

\bibitem[Cai and Low(2006)]{CaiLow2006}
T.~T. Cai and M.~G. Low.
\newblock Optimal adaptive estimation of a quadratic functional.
\newblock \emph{The Annals of Statistics}, 34\penalty0 (5):\penalty0
  2298--2325, 2006.

\bibitem[Chesneau(2011)]{Chesneau2011}
C.~Chesneau.
\newblock On adaptive wavelet estimation of a quadratic functional from a
  deconvolution problem.
\newblock \emph{Annals of the Institute of Statistical Mathematics},
  63\penalty0 (2):\penalty0 405--429, 2011.

\bibitem[Collier et~al.(2017)Collier, Comminges, and
  Tsybakov]{CollierCommingesTsybakov2017}
O.~Collier, L.~Comminges, and A.~B. Tsybakov.
\newblock Minimax estimation of linear and quadratic functionals on sparsity
  classes.
\newblock \emph{The Annals of Statistics}, 45\penalty0 (3):\penalty0 923--958,
  2017.

\bibitem[Comte and Taupin(2003)]{ComteTaupin2003}
F.~Comte and M.-L. Taupin.
\newblock \emph{Adaptive density deconvolution for circular data}.
\newblock 2003.

\bibitem[Efromovich(1997)]{Efromovich1997}
S.~Efromovich.
\newblock Density estimation for the case of supersmooth measurement error.
\newblock \emph{Journal of the American Statistical Association}, 92\penalty0
  (438):\penalty0 526--535, 1997.

\bibitem[Fisher(1995)]{Fisher1995}
N.~I. Fisher.
\newblock \emph{Statistical analysis of circular data}.
\newblock Cambridge University Press, 1995.

\bibitem[Gill and Hangartner(2010)]{GillHangartner2010}
J.~Gill and D.~Hangartner.
\newblock Circular data in political science and how to handle it.
\newblock \emph{Political Analysis}, 18\penalty0 (3):\penalty0 316--336, 2010.

\bibitem[Ingster(1993{\natexlab{a}})]{Ingster1993}
Y.~Ingster.
\newblock Asymptotically minimax hypothesis testing for nonparametric
  alternatives {I}.
\newblock \emph{Mathematical Methods of Statistics}, 2\penalty0 (2):\penalty0
  85--114, 1993{\natexlab{a}}.

\bibitem[Ingster(1993{\natexlab{b}})]{Ingster1993a}
Y.~Ingster.
\newblock Asymptotically minimax hypothesis testing for nonparametric
  alternatives {II}.
\newblock \emph{Mathematical Methods of Statistics}, 2\penalty0 (2):\penalty0
  171--189, 1993{\natexlab{b}}.

\bibitem[Ingster(1993{\natexlab{c}})]{Ingster1993b}
Y.~Ingster.
\newblock Asymptotically minimax hypothesis testing for nonparametric
  alternatives {III}.
\newblock \emph{Mathematical Methods of Statistics}, 2\penalty0 (2):\penalty0
  249---268, 1993{\natexlab{c}}.

\bibitem[Johannes and Schwarz(2013)]{JohannesSchwarz2013a}
J.~Johannes and M.~Schwarz.
\newblock Adaptive circular deconvolution by model selection under unknown
  error distribution.
\newblock \emph{Bernoulli}, 19\penalty0 (5A):\penalty0 1576--1611, 2013.

\bibitem[Kerkyacharian et~al.(2011)Kerkyacharian, Pham~Ngoc, and
  Picard]{KerkyacharianNgocPicard2011a}
G.~Kerkyacharian, T.~M. Pham~Ngoc, and D.~Picard.
\newblock Localized spherical deconvolution.
\newblock \emph{The Annals of Statistics}, 39\penalty0 (2):\penalty0
  1042--1068, 04 2011.

\bibitem[Kroll(2019)]{Kroll2019}
M.~Kroll.
\newblock Rate optimal estimation of quadratic functionals in inverse problems
  with partially unknown operator and application to testing problems.
\newblock \emph{ESAIM: Probability and Statistics}, 23:\penalty0 524--551,
  2019.

\bibitem[Lacour and Ngoc(2014)]{LacourNgoc2014}
C.~Lacour and T.~M.~P. Ngoc.
\newblock Goodness-of-fit test for noisy directional data.
\newblock \emph{Bernoulli}, 20\penalty0 (4):\penalty0 2131--2168, 2014.

\bibitem[Laurent(2005)]{Laurent2005}
B.~Laurent.
\newblock Adaptive estimation of a quadratic functional of a density by model
  selection.
\newblock \emph{ESAIM: Probability and Statistics}, 9:\penalty0 1--18, 2005.

\bibitem[Laurent and Massart(2000)]{LaurentMassart2000}
B.~Laurent and P.~Massart.
\newblock Adaptive estimation of a quadratic functional by model selection.
\newblock \emph{The Annals of Statistics}, 28\penalty0 (5):\penalty0
  1302--1338, 2000.

\bibitem[Laurent et~al.(2012)Laurent, Loubes, and
  Marteau]{LaurentLoubesMarteau2012}
B.~Laurent, J.-M. Loubes, and C.~Marteau.
\newblock Non asymptotic minimax rates of testing in signal detection with
  heterogeneous variances.
\newblock \emph{Electronic Journal of Statistics}, 6:\penalty0 91--122, 2012.

\bibitem[Mardia(1972)]{Mardia1972}
K.~V. Mardia.
\newblock \emph{Statistics of directional data}.
\newblock Academic press, 1972.

\bibitem[Mardia and Jupp(2009)]{MardiaJupp2009}
K.~V. Mardia and P.~E. Jupp.
\newblock \emph{Directional statistics}, volume 494.
\newblock John Wiley \& Sons, 2009.

\bibitem[Marteau and Sapatinas(2017)]{MarteauSapatinas2017}
C.~Marteau and T.~Sapatinas.
\newblock Minimax goodness-of-fit testing in ill-posed inverse problems with
  partially unknown operators.
\newblock \emph{Annales de l'Institut Henri Poincar{\'e}, Probabilit{\'e}s et
  Statistiques}, 53\penalty0 (4):\penalty0 1675--1718, 2017.

\bibitem[Serfling(2009)]{Serfling2009}
R.~J. Serfling.
\newblock \emph{Approximation theorems of mathematical statistics}, volume 162.
\newblock John Wiley \& Sons, 2009.

\bibitem[Tsybakov(2009)]{Tsybakov2009}
A.~B. Tsybakov.
\newblock \emph{Introduction to Nonparametric Estimation}.
\newblock Springer New York, 2009.

\end{thebibliography}
\end{document}